\documentclass[reqno,english]{amsart}
\usepackage{amsmath,amssymb,amsthm,amsfonts,mathrsfs,array,booktabs}
\usepackage{enumitem,needspace}
\usepackage[expansion=false]{microtype}
\usepackage{hyperref}
\hypersetup{colorlinks=true,linkcolor=black,citecolor=black,urlcolor=black}
\numberwithin{equation}{section}
\newtheorem{theorem}{Theorem}[section]
\newtheorem{lemma}{Lemma}[section]
\newtheorem{proposition}{Proposition}[section]
\newtheorem{corollary}{Corollary}[section]
\theoremstyle{remark}

\newcommand{\R}{\mathbb R}
\newcommand{\Z}{\mathbb Z}
\newcommand{\F}{L_{\mathrm F}}
\newcommand{\B}{L_{\mathrm B}}
\newcommand{\Id}{I}
\newcommand{\op}{\mathrm{op}}
\newcommand{\fro}{\mathrm F}
\DeclareMathOperator{\tr}{tr}
\DeclareMathOperator{\erfc}{erfc}
\DeclareMathOperator*{\argmin}{arg\,min}
\allowdisplaybreaks[2]
\let\oldbibliography\thebibliography
\renewcommand{\thebibliography}[1]{\oldbibliography{#1}\setlength{\itemsep}{0pt}}
\hypersetup{pdftitle={On exact solutions to the three-dimensional Yukawa lattice model and the BCC--FCC transition},pdfauthor={Senping Luo and Juncheng Wei}}
\title[The Yukawa lattice model]{\bf On exact solutions to three-dimensional Yukawa lattice model and BCC to FCC phase transitions}
\author{Senping Luo}

\author{Juncheng Wei}

\address[S.~Luo]{School of Mathematics and Statistics, Jiangxi Normal University, Nanchang, 330022, China}
\address[J.~Wei]{Department of Mathematics, Chinese University of Hong Kong, Shatin, NT, Hong
Kong}

\email[S.~Luo]{luosp1989@163.com}

\email[J.~Wei]{wei@math.cuhk.edu.hk}

\date{}
\begin{document}
\begin{abstract}
We determine the global minimizers of the three-dimensional Yukawa
lattice energy at every fixed density. At unit covolume, BCC is the
unique minimizer below a transition parameter $\kappa_*\approx1.7179218684699$, FCC is the
unique minimizer above it, and exactly these two lattice classes
minimize at the transition.

 Mathematically, denote the Yukawa potential and the energy per particle under this potential by
\[
 v_\kappa(r)=\frac{e^{-\kappa r}}r,\qquad
 E_\kappa(L)=\frac12\sum_{v\in L\setminus\{0\}}
                   \frac{e^{-\kappa|v|}}{|v|}.
\]
Then,
\[
 \arg\min_{|L|=1}E_\kappa(L)=
 \begin{cases}
 \hbox{BCC lattice},&0<\kappa<\kappa_*,\\
 \hbox{BCC or FCC lattice},&\kappa=\kappa_*,\\
 \hbox{FCC lattice},&\kappa>\kappa_*.
 \end{cases}
\]
 As a consequence, the zero-screening comparison recovers the unique BCC
minimum for the renormalized Coulomb energy.
\end{abstract}
\maketitle
\enlargethispage{4pt}

\section{Introduction and main results}
The Yukawa potential describes a screened Coulomb interaction and has been studied extensively since 1935 \cite{Yukawa}, in which Yukawa combined Heisenberg's idea of a short-range force interaction and Fermi's idea of an exchange particle in order to resolve the problem of the neutron-proton interaction. Its
screening parameter introduces a competition between two cubic lattice
structures: numerical phase diagrams and lattice calculations predict
BCC at weak screening and FCC at strong screening~\cite{RKG,HFD96,HFD,Koz}.
The transition has also been studied through applied-strain free-energy
calculations~\cite{HR}. The purpose of this paper is to determine the
global minimizing Bravais lattices for every positive screening parameter
at zero temperature and to prove that there is exactly one transition.

For $\kappa>0$, the potential and the energy per particle are
\begin{equation}\label{eq:potential}
 v_\kappa(r)=\frac{e^{-\kappa r}}r,\qquad
 E_\kappa(L)=\frac12\sum_{v\in L\setminus\{0\}}
                   \frac{e^{-\kappa|v|}}{|v|}.
\end{equation}
The history of this potential involves some physical settings.
Debye and H\"uckel obtained the screened Coulomb form in 1923
in their theory of electrolyte solutions~\cite{DebyeHuckel}.
Mobile ions rearrange around a charged particle, forming an ionic
atmosphere that reduces its electrostatic influence at large
distances.

Mathematically, the normalized kernel (Yukawa potential)
\[
 G_\kappa(x)=\frac{e^{-\kappa|x|}}{4\pi|x|}
\]
is the fundamental solution of the modified Helmholtz operator:
\[
 (-\Delta+\kappa^2)G_\kappa=\delta_0
 \qquad\text{in }\mathbb R^3.
\]
This equation explains the combination of a Coulomb singularity
and exponential spatial decay.

The factor $1/2$ in \eqref{eq:potential} counts each pair once. For a lattice $L=g\mathbb Z^3$,
write $|L|=|\det g|$ for its covolume. We first impose $|L|=1$
and identify lattices that differ by an orthogonal transformation.
Write $L^*=g^{-T}\mathbb Z^3$ and
$\lambda_1(L)=\min_{v\in L\setminus\{0\}}|v|$.
The unit-covolume face-centred and body-centred cubic lattices are
\begin{equation}\label{eq:FCC}
 D_3=\{m\in\mathbb Z^3:m_1+m_2+m_3\in2\mathbb Z\},\qquad
 \F=2^{-1/3}D_3,\qquad \B=\F^*.
\end{equation}
Here $\F$ is the face-centred cubic (FCC) lattice and
$\B$ is the body-centred cubic (BCC) lattice, with $|\F|=|\B|=1$.
We shall use $a=\lambda_1(\F)=2^{1/6}$. Standard facts about these
lattices can be found in~\cite{CS}.

For $t>0$, define
\begin{equation}\label{eq:theta}
 \Theta(t,L)=\sum_{v\in L}e^{-\pi t|v|^2}.
\end{equation}

\Needspace{36\baselineskip}
\begin{theorem}[Minima of three-dimensional Yukawa lattice energy]\label{thm:main}
Let $L\subset\mathbb R^3$ be a lattice with $|L|=1$, and let
$\kappa>0$.
There is a unique threshold $\kappa_*\in(0,\sqrt\pi)$,
characterized by
\[
 E_{\kappa_*}(\B)=E_{\kappa_*}(\F).
\]
More explicitly, $\kappa_*$ is the unique positive solution of
\begin{equation}\label{eq:root}
 \int_1^\infty
 \bigl(\Theta(t,\B)-\Theta(t,\F)\bigr)
 \left[t^{-1/2}e^{-\kappa_*^2/(4\pi t)}
             -e^{-\kappa_*^2t/(4\pi)}\right]dt=0.
\end{equation}
Numerically,
$$\mbox{$\kappa_*\approx1.71792186846990$}.
$$
The following statements hold.
\begin{enumerate}[label=(\arabic*),leftmargin=*]
\item If $\kappa>\kappa_*$, then
      $E_\kappa(L)\ge E_\kappa(\F)$, with equality if and only if
      $L$ is orthogonally equivalent to $\F$, corresponding to FCC lattice.
\item If $0<\kappa<\kappa_*$, then
      $E_\kappa(L)\ge E_\kappa(\B)$, with equality if and only if
      $L$ is orthogonally equivalent to $\B$, corresponding to BCC lattice.
\item If $\kappa=\kappa_*$, then
      $E_\kappa(L)\ge E_\kappa(\F)=E_\kappa(\B)$, with equality if and
      only if $L$ is orthogonally equivalent to $\F$ or $\B$, corresponding to FCC or BCC lattice.
\end{enumerate}
Equivalently, up to orthogonal transformations,
\begin{equation}\label{eq:argmin}
 \argmin_{|L|=1}E_\kappa(L)=
 \begin{cases}
 \{\B\},&0<\kappa<\kappa_*,\\
 \{\B,\F\},&\kappa=\kappa_*,\\
 \{\F\},&\kappa>\kappa_*.
 \end{cases}
\end{equation}
At density $\varrho>0$, the minimizing lattices are the corresponding
FCC and BCC lattices multiplied by $\varrho^{-1/3}$, and the transition occurs at
$\kappa=\varrho^{1/3}\kappa_*$.
\end{theorem}

The density dependence follows from
\begin{equation}\label{eq:scaling}
 E_\kappa(\varrho^{-1/3}L)
 =\varrho^{1/3}E_{\kappa\varrho^{-1/3}}(L).
\end{equation}
Throughout the paper, minimization is over Bravais lattices of fixed
density, with one particle in each primitive cell.

The comparison underlying Theorem~\ref{thm:main} also includes zero
screening. It therefore yields the following Coulomb consequence.
Here $Z(L,s)$ denotes the meromorphic continuation of
$\sum_{v\in L\setminus\{0\}}|v|^{-2s}$, initially defined for
$\Re s>3/2$.

\Needspace{11\baselineskip}
\begin{corollary}[Coulomb minimizers]\label{thm:coulomb-endpoint-certified}
For every lattice $L$ with $|L|=1$,
\begin{equation}\label{eq:coulomb-consequence}
 \lim_{\kappa\downarrow0}
 \left(2E_\kappa(L)-\frac{4\pi}{\kappa^2}\right)
 =Z(L,1/2)\ge Z(\B,1/2),
\end{equation}
with equality in the inequality if and only if $L$ is orthogonally
equivalent to $\B$. Equivalently, $Z(M,1)\ge Z(\F,1)$ for every
lattice $M$ with $|M|=1$, with equality if and only if $M$ is
orthogonally equivalent to $\F$.
\end{corollary}

Muratov and Novaga~\cite[Theorem~1]{MN} proved that BCC uniquely
minimizes the renormalized three-dimensional Coulomb energy among
Bravais lattices of fixed covolume. Their energy is exactly the limit
in~\eqref{eq:coulomb-consequence}. Thus Corollary~
\ref{thm:coulomb-endpoint-certified} recovers their theorem, including
its equality case, as the zero-screening consequence of our uniform
Yukawa comparison. In their Madelung normalization,
\[
 K_L=\frac12\left(\frac3{4\pi}\right)^{1/3}Z(L,1/2)
 \qquad(|L|=1).
\]
The proof below obtains uniqueness from the strict comparison at zero
screening and the strict FCC--BCC energy gap.

The lattice minimization problem belongs to the broader theory of
crystallization; see the reviews of Blanc--Lewin~\cite{BL} and
Lewin~\cite{Lewin}.

Fourier analysis provides another approach to energy minimization.
Cohn--Elkies~\cite{CE} developed linear programming bounds for the
density of sphere packings. Cohn--Kumar~\cite{CK} proved universal
optimality for sharp configurations on spheres.
Cohn, Kumar, Miller, Radchenko and Viazovska~\cite{CKMRV} subsequently
proved universal optimality of the $E_8$ and Leech lattices at fixed
density in dimensions $8$ and $24$, respectively, for completely
monotone functions of squared distance. The Yukawa potential belongs to this
class: its positive Gaussian representation~\eqref{eq:gaussian}
shows that it is completely monotone as a function of squared distance.

The Gaussian representation also connects the present problem with
the Gaussian core model. Cohn, Kumar and Sch\"urmann~\cite{CKS}
numerically investigated its ground states and identified formal
duality relations. Cohn--de Courcy-Ireland~\cite{CC}
established asymptotically sharp energy bounds for Gaussian interactions
in a range of parameters, at fixed density as the dimension tends to
infinity.

Renormalized energy connects these problems with the microscopic theory
of Coulomb and Riesz gases. Sandier--Serfaty~\cite{SSV,SSGas} develop
this approach for planar vortices and logarithmic gases.
Rougerie--Serfaty~\cite{RS} and Petrache--Serfaty~\cite{PSRiesz}
identify microscopic renormalized energies in next-order Coulomb and
Riesz asymptotics. Lebl\'e--Serfaty~\cite{LS} prove an empirical-field
large-deviation principle, while Armstrong--Serfaty~\cite{AS} obtain
local laws and rigidity estimates. Petrache--Serfaty~\cite{PSCrystal}
connect universal optimality with renormalized Coulomb and Riesz
crystallization, giving unconditional minimization results in
dimensions $8$ and $24$. Serfaty's monographs~\cite{SerBook,SerLect}
provide systematic accounts of these developments.

For three-dimensional lattice energies,
the work of Sarnak--Str\"ombergsson~\cite{SS} on Epstein zeta functions
and heights of flat tori provides a fundamental analytic framework.
B\'etermin--Petrache~\cite{BP} develop dimension-reduction methods for
lattice theta functions. Related Fourier-duality identities for soft
pair potentials are obtained by Torquato--Stillinger~\cite{TS}.
Classical local optimality results are due to Ennola~\cite{Enn};
B\'etermin~\cite{Bet} treats local variations for general interaction
potentials. Three-dimensional Riesz and Lennard--Jones lattice energies
are studied in~\cite{BST}. Ren--Wei~\cite{RW} establish BCC selection
for a different long-range interaction.

Our argument follows the local--global method developed in our previous
work~\cite{LW}, extending the framework of Sarnak--Str\"ombergsson~\cite{SS}.
Our analysis of theta differences~\cite{LWTheta} includes the
hexagonal Yukawa minimum in two dimensions. Sharp lattice transitions
for competing Epstein zeta energies and for theta derivatives are
studied in~\cite{LWJEMS,LWMA}, respectively. The present problem
requires a comparison in all five shape variables of a three-dimensional
Bravais lattice, uniformly in the screening parameter. The unique
Yukawa transition is determined by integrating the FCC--BCC theta
difference, rather than by a pointwise theta comparison.

The global comparison reduces minimization to FCC and BCC for every
screening parameter. We then prove that their energies cross exactly
once. Put $c=\kappa^2/(4\pi)$.
Ewald splitting~\cite{Ewald} gives a positive functional $P(L,c)$,
defined in~\eqref{eq:positive-P}, with the same lattice differences as
$E_\kappa$. We prove
\begin{equation}\label{eq:intro-comparison}
 P(L,c)\ge\min\{P(\B,c),P(\F,c)\}
 \qquad(c\ge0),
\end{equation}
with equality only when $L$ is orthogonally equivalent to $\B$ or $\F$.
A shortest-vector
estimate first excludes the cusp. An elementary lattice reduction then
places every remaining competitor in a fixed compact rectangle in five
coordinates. Near each reference, cubic symmetry reduces the Hessian
and third variation to finitely many scalar coefficients. A uniform
fourth-variation estimate gives a fixed neighbourhood on which the
appropriate local comparison is strict.

Outside these neighbourhoods, convexity of the positive Ewald kernels
reduces the shape comparison to the $32$ corners of each rectangle.
We subtract the reference before expanding in $c$ and bound the
resulting polynomial by its Bernstein coefficients. This preserves
the cancellation near the transition. For large screening, the
normalized energy $2ae^{\kappa a}E_\kappa$ admits a convexity argument
that extends a finite comparison to the whole remaining half-line.
Finally, Poisson duality and the scalar FCC--BCC theta ordering prove
that the two reference energies cross exactly once.

The proof is computer assisted. Appendix~\ref{app:numerics} collects
the independent numerical comparisons; the complete interval
computations and exact partitions accompany the manuscript.

Section~\ref{sec:representation} establishes the integral representations,
duality identities, and Coulomb limit; Section~\ref{sec:reduction}
gives the compact reduction. Sections~\ref{sec:certified-local}
and~\ref{sec:screening-profile-lemma} prove the local and global
comparisons. Section~\ref{sec:pair} determines the transition and
completes the main proofs. Appendix~\ref{app:numerics} contains the
numerical lemmas.

\section{Integral representation and duality}\label{sec:representation}
We express the Yukawa energy through positive Gaussian kernels and
separate its shape-dependent part. The same representation yields an
exact Coulomb limit and compactness uniform in weak screening.

\begin{theorem}\label{thm:representation}
For every $\kappa>0$ and every lattice $L$ with $|L|=1$,
\begin{equation}\label{eq:gaussian}
 2E_\kappa(L)=\int_0^\infty
 t^{-1/2}e^{-\kappa^2/(4\pi t)}\bigl(\Theta(t,L)-1\bigr)\,dt.
\end{equation}
The energy is continuous, its sublevel sets are compact, and its
minimum is attained.
\end{theorem}
The Gaussian identity permits Poisson summation, while the singularity
of the potential excludes degenerating lattices.
\begin{proof}
For $A,B>0$,
\[
 \int_0^\infty t^{-1/2}e^{-At-B/t}\,dt
 =\sqrt{\frac\pi A}\,e^{-2\sqrt{AB}}.
\]
Take $A=\pi r^2$ and $B=\kappa^2/(4\pi)$, and apply Tonelli's theorem.
In a lattice chart $L=g\Z^3$ on which $g,g^{-1}$ are bounded, the
summands of $E_\kappa$ are bounded by a constant multiple of
$e^{-\epsilon\kappa|m|}/|m|$ for some $\epsilon>0$.
Normal convergence gives continuity. A shortest opposite pair gives
\begin{equation}\label{eq:shortest}
 E_\kappa(L)\ge
 \frac{e^{-\kappa\lambda_1(L)}}{\lambda_1(L)}.
\end{equation}
Thus $\lambda_1(L)$ is bounded away from zero on every sublevel set.
Mahler's compactness theorem~\cite{Cassels} makes each sublevel
relatively compact, and continuity makes it closed. Hence every
minimizing sequence has a convergent subsequence whose limit attains
the minimum.
\end{proof}

\begin{proposition}\label{prop:dual}
Put $c=\kappa^2/(4\pi)$. For lattices $L,R$ with $|L|=|R|=1$,
\begin{align}\label{eq:dual-gap}
 2\bigl(E_\kappa(L)-E_\kappa(R)\bigr)
 ={}&\int_1^\infty t^{-1/2}e^{-c/t}
       \bigl(\Theta(t,L)-\Theta(t,R)\bigr)\,dt\notag\\
 &+\int_1^\infty e^{-ct}
       \bigl(\Theta(t,L^*)-\Theta(t,R^*)\bigr)\,dt.
\end{align}
Both integrals converge absolutely.
\end{proposition}
\begin{proof}
Split \eqref{eq:gaussian} at one. On $(0,1)$ use
$\Theta(t,L)=t^{-3/2}\Theta(1/t,L^*)$ and substitute $t\mapsto1/t$.
The zero terms cancel in the difference, and Gaussian decay proves
absolute convergence.
\end{proof}

For $q>0$ and $c\ge0$, define
\[
 I(q,c)=\int_1^\infty t^{-1/2}e^{-\pi qt-c/t}\,dt,\qquad
 J(q,c)=\int_1^\infty e^{-(\pi q+c)t}\,dt
       =\frac{e^{-\pi q-c}}{\pi q+c}.
\]
For a determinant-one Gram matrix $Y=g^Tg$, write $Y[m]=m^TYm$ and put
\begin{equation}\label{eq:positive-P}
 P(Y,c)=\frac12\sum_{m\in\Z^3\setminus\{0\}}
       \{I(Y[m],c)+J(Y^{-1}[m],c)\}.
\end{equation}
We also write $P(L,c)$ for this basis-independent functional. Every
summand is positive, and both kernels are completely monotone in $q$.
Their sums and any fixed number of shape derivatives converge normally
on compact sets of nonsingular lattices and bounded screening intervals.
These assertions follow by differentiation under the positive integrals
and Gaussian domination of the resulting polynomial factors.

Let $Z(L,s)$ denote the meromorphic continuation of
$\sum_{v\in L\setminus\{0\}}|v|^{-2s}$ from $\Re s>3/2$, and define
\begin{equation}\label{eq:Coulomb}
 W(L)=\frac12Z(L,1/2)=\frac1{2\pi}Z(L^*,1),\qquad
 F_\kappa(L)=E_\kappa(L)-\frac{2\pi}{\kappa^2}-\frac\kappa2,
 \quad F_0=W.
\end{equation}
The identity for $W$ is the Epstein functional equation; the displayed
zeta values are continued values, not divergent lattice sums.

\begin{proposition}\label{prop:ewald}
For $\kappa>0$ and $c=\kappa^2/(4\pi)$,
\begin{equation}\label{eq:positive-ewald}
 2F_\kappa(L)=2P(L,c)+\frac{e^{-c}-1}{c}-2e^{-c}
                 -\kappa\operatorname{erf}(\sqrt c).
\end{equation}
The last three terms tend to $-3$ as $\kappa\downarrow0$, and
\begin{equation}\label{eq:P-difference}
 P(L,c)-P(R,c)=F_\kappa(L)-F_\kappa(R),\qquad
 W(L)=P(L,0)-\frac32.
\end{equation}
Thus $P(\cdot,c)$ and $E_\kappa$ have the same minimizers when
$\kappa>0$.
\end{proposition}
\begin{proof}
Splitting \eqref{eq:gaussian} and using Poisson summation gives
\[
 2E_\kappa(L)=2P(L,c)+\frac{e^{-c}}c
                   -\int_0^1t^{-1/2}e^{-c/t}\,dt.
\]
With $u=\sqrt{c/t}$, integration by parts gives
\[
 \int_0^1t^{-1/2}e^{-c/t}\,dt
 =2\sqrt c\int_{\sqrt c}^\infty\frac{e^{-u^2}}{u^2}\,du
 =2e^{-c}-\kappa\erfc(\sqrt c).
\]
Since $2F_\kappa=2E_\kappa-1/c-\kappa$ and
$\erfc(\sqrt c)=1-\operatorname{erf}(\sqrt c)$,
subtraction gives \eqref{eq:positive-ewald}.
The regularized Mellin formula at $s=1/2$ is
\[
 Z(L,1/2)=\int_1^\infty t^{-1/2}(\Theta(t,L)-1)\,dt
        +\int_1^\infty(\Theta(t,L^*)-1)\,dt-3.
\]
It proves the assertion at zero screening and the difference identity.
\end{proof}

The direct kernel also has the elementary form
\[
 I(q,c)=\frac{e^{-\kappa\sqrt q}
          \erfc(\sqrt{\pi q}-\sqrt c)
        +e^{\kappa\sqrt q}\erfc(\sqrt{\pi q}+\sqrt c)}{2\sqrt q},
 \qquad \kappa=\sqrt{4\pi c}.
\]
A shortest opposite pair contributes $I(\lambda_1(L)^2,c)$ to
$P(L,c)$. Since $e^{-c/t}\ge e^{-c}$ for $t\ge1$,
\[
 P(L,c)\ge I(\lambda_1(L)^2,c)
 \ge e^{-c}I(\lambda_1(L)^2,0)
 =\frac{e^{-c}\erfc(\sqrt\pi\lambda_1(L))}{\lambda_1(L)}.
\]
Substitution in \eqref{eq:positive-ewald} gives
\begin{equation}\label{eq:uniformcusp}
 F_\kappa(L)\ge
 \frac{e^{-c}\erfc(\sqrt\pi\lambda_1(L))}{\lambda_1(L)}
 +\frac{e^{-c}-1}{2c}-e^{-c}
 -\frac\kappa2\operatorname{erf}(\sqrt c).
\end{equation}
The right side extends continuously to $\kappa=0$ and tends to
infinity as $\lambda_1(L)\to0$, uniformly on bounded screening
intervals. This uniform exclusion of degenerating lattices permits
passage from the Coulomb limit of each energy to the limit of the
global minima.

\begin{proposition}[The Coulomb limit]\label{thm:weak}
For every lattice $L$ with $|L|=1$ and every $\kappa>0$,
\begin{equation}\label{eq:resolvent}
 F_\kappa(L)=W(L)-\frac{\kappa^2}{8\pi^3}
 \sum_{w\in L^*\setminus\{0\}}
 \frac1{|w|^2\bigl(|w|^2+\kappa^2/(4\pi^2)\bigr)}.
\end{equation}
The correction converges absolutely, and
\begin{equation}\label{eq:weakexpansion}
 0\le F_\kappa(L)-W(L)
       +\frac{\kappa^2}{8\pi^3}Z(L^*,2)
 \le\frac{\kappa^4}{32\pi^5}Z(L^*,3).
\end{equation}
In particular $F_\kappa\to W$ locally uniformly as
$\kappa\downarrow0$. The functional $W$ attains its minimum, and
\begin{equation}\label{eq:weakminimum}
 \lim_{\kappa\downarrow0}
 \left\{\min_{|L|=1}E_\kappa(L)
              -\frac{2\pi}{\kappa^2}-\frac\kappa2\right\}
 =\min_{|L|=1}W(L).
\end{equation}
Every sequence of Yukawa minimizers with $\kappa\downarrow0$ has a
convergent subsequence, and every such limit minimizes $W$.
\end{proposition}
\begin{proof}
Put $Q_L(t)=\Theta(t,L)-1-t^{-3/2}$. Poisson summation shows that
$Q_L(t)=-1+O(t^{-3/2}e^{-d/t})$ at zero and
$Q_L(t)=-t^{-3/2}+O(e^{-dt})$ at infinity, for some $d>0$.
Hence $t^{-1/2}Q_L(t)$ is integrable, and the regularized Mellin formula
reads $Z(L,1/2)=\int_0^\infty t^{-1/2}Q_L(t)\,dt$.
Since $\int_0^\infty t^{-2}e^{-c/t}\,dt=1/c$, we also have
\begin{equation}\label{eq:subtracted-integral}
 2E_\kappa(L)-\frac1c
 =\int_0^\infty t^{-1/2}e^{-c/t}Q_L(t)\,dt.
\end{equation}
Poisson summation gives the explicit identity
\[
 u^{-3/2}Q_L(1/u)=\Theta(u,L^*)-1-u^{-3/2}.
\]
Subtract the Mellin identity from \eqref{eq:subtracted-integral}
and put $t=1/u$. The weighted measure transforms as
$t^{-1/2}\,dt=-u^{-3/2}\,du$, so
\begin{align*}
 2E_\kappa(L)-\frac1c-Z(L,1/2)
 &=\int_0^\infty(e^{-cu}-1)
          (\Theta(u,L^*)-1-u^{-3/2})\,du\\
 &=\sum_{w\in L^*\setminus\{0\}}
   \left\{\frac1{\pi|w|^2+c}-\frac1{\pi|w|^2}\right\}+\kappa.
\end{align*}
After separating $u^{-3/2}$, the theta integral is absolutely
convergent: its integrand is $O(u^{-1/2})$ at zero and decays
exponentially at infinity. Summation and integration therefore commute.
The remaining term is
$\int_0^\infty(1-e^{-cu})u^{-3/2}\,du
 =2\sqrt{\pi c}=\kappa$.
Thus $F_\kappa(L)-W(L)$ is one half of the reciprocal sum above.
The identity
\[
 \frac12\left(\frac1{\pi|w|^2+c}-\frac1{\pi|w|^2}\right)
 =-\frac{c}{2\pi^2|w|^2\bigl(|w|^2+c/\pi\bigr)},
\]
together with $c/\pi=\kappa^2/(4\pi^2)$ and
$c/(2\pi^2)=\kappa^2/(8\pi^3)$, proves \eqref{eq:resolvent}.

The identity
\[
 \frac1{q(q+x)}=\frac1{q^2}-\frac{x}{q^2(q+x)}
 \qquad(q,x>0)
\]
gives
\[
 F_\kappa(L)-W(L)+\frac{\kappa^2}{8\pi^3}Z(L^*,2)
 =\frac{\kappa^4}{32\pi^5}
 \sum_{w\in L^*\setminus\{0\}}
 \frac1{|w|^4\bigl(|w|^2+\kappa^2/(4\pi^2)\bigr)}.
\]
Since $|w|^2+\kappa^2/(4\pi^2)\ge|w|^2$, this proves
\eqref{eq:weakexpansion}.
Normal convergence bounds $Z(L^*,2)$ and $Z(L^*,3)$ on compact sets of
lattice classes, proving locally uniform convergence.

For any fixed $\kappa_0>0$, the right side of
\eqref{eq:uniformcusp} tends to infinity as $\lambda_1(L)\to0$,
uniformly for $0\le\kappa\le\kappa_0$: the lattice-independent
terms in \eqref{eq:uniformcusp} extend continuously to $c=0$, and
$e^{-c}$ is bounded below by a positive constant. Thus $W$ has compact sublevels
by Mahler's theorem and attains a minimum. Moreover, a minimizer
$L_\kappa$ of $F_\kappa$ satisfies
\[
 F_\kappa(L_\kappa)\le F_\kappa(\B)\le W(\B),
\]
where the last inequality follows from \eqref{eq:resolvent}.
Consequently all these minimizers, for
$0<\kappa\le\kappa_0$, lie in one compact set.

If $\kappa_j\downarrow0$, choose a convergent subsequence of minimizers
with limit $L_0$. Locally uniform convergence gives
$F_{\kappa_j}(L_{\kappa_j})\to W(L_0)$. For any fixed competitor $R$,
$F_{\kappa_j}(L_{\kappa_j})\le F_{\kappa_j}(R)$; passing to the limit
shows that $L_0$ minimizes $W$. Every sequence admits such a subsequence,
and its limiting minimum value is always $\min W$.
This proves \eqref{eq:weakminimum} and the assertion about minimizers.
\end{proof}

\section{An elementary compact reduction}\label{sec:reduction}

We reduce the global problem to an explicit compact set of lattice
shapes. The reduction is uniform in the screening parameter, including
the continued Coulomb endpoint. Throughout this section $c\ge0$ denotes
the screening variable in the positive Ewald functional $P(L,c)$.

\Needspace{18\baselineskip}
\begin{proposition}[Compact reduction]\label{prop:compact-reduction}
For every $c\ge0$, the functional $P(\cdot,c)$ attains its minimum
among lattices $L$ with $|L|=1$. Every global minimizer satisfies
$\lambda_1(L)>4/5$ and has a Gram matrix of the form
\begin{equation}\label{eq:reduced-iwasawa}
 Y=N^TDN,\qquad
 N=\begin{pmatrix}1&x_1&x_2\\0&1&x_3\\0&0&1\end{pmatrix},
 \qquad D=\operatorname{diag}(d_1,d_2,d_3),
\end{equation}
where, on putting $y_i=e^{\ell_i}$,
\[
 d_1=y_1^{2/3}y_2^{1/3},\qquad
 d_2=y_1^{-1/3}y_2^{1/3},\qquad
 d_3=y_1^{-1/3}y_2^{-2/3},
\]
and
\begin{equation}\label{eq:endpoint-rational-root}
 (y_1,y_2,x_1,x_2,x_3)\in
 \left[\frac{11}{25},\frac43\right]
 \times\left[\frac{147}{1000},\frac43\right]
 \times\left[0,\frac12\right]\times[0,1]
 \times\left[0,\frac12\right].
\end{equation}
In the basis constructed below, $d_1=\lambda_1(L)^2$.
\end{proposition}

We first choose a shortest-vector basis and reduce its two-dimensional
orthogonal projection. A comparison of the Ewald kernel ratios then
excludes the cusp for every $c\ge0$. The resulting rectangle retains
all five shape parameters; it does not restrict the problem to a
symmetric family of lattices. We also record the separate convexity
which will control every shape rectangle by its corners.

\begin{lemma}[Elementary covering reduction]\label{lem:elementary-reduction}
Every lattice $L$ with $|L|=1$ has a basis with Gram matrix
\eqref{eq:reduced-iwasawa} satisfying
$0\le x_1,x_3\le1/2$, $0\le x_2\le1$, and
\begin{align}\label{eq:reduction-inequalities}
 1&\le(1+x_1-x_2)^2+y_1^{-1}\bigl((1-x_3)^2+y_2^{-1}\bigr),\notag\\
 1&\le(x_1-x_2)^2+y_1^{-1}\bigl((1-x_3)^2+y_2^{-1}\bigr),\notag\\
 1&\le x_1^2+y_1^{-1},\notag\\
 1&\le x_2^2+y_1^{-1}(x_3^2+y_2^{-1}),\\
 1&\le(x_2-1)^2+y_1^{-1}(x_3^2+y_2^{-1}),\notag\\
 1&\le x_3^2+y_2^{-1}.\notag
\end{align}
Its first basis vector is shortest, so $d_1=\lambda_1(L)^2$.
\end{lemma}
\begin{proof}
Choose a shortest nonzero vector $v_1\in L$. It is primitive and hence
extends to an integral basis. If $p$ is orthogonal projection onto
$v_1^\perp$, then $p(L)$ is a rank-two lattice and
$\ker(p|_L)=\mathbb Zv_1$. Indeed, projection of the other two vectors
of any basis beginning with $v_1$ gives two linearly independent
vectors generating $p(L)$.
Choose a shortest nonzero vector $w_2\in p(L)$. It is primitive, so it
extends to a basis $(w_2,w_3)$ of $p(L)$. Choose lifts $v_2,v_3\in L$.
Then $(v_1,v_2,v_3)$ is a basis of $L$: projecting an arbitrary vector
first expresses its image in terms of $w_2,w_3$, and the remaining
vector belongs to $\mathbb Zv_1$.

Gram--Schmidt gives an orthonormal frame $(e_1,e_2,e_3)$ and positive
numbers $d_i$ such that
\[
 v_1=\sqrt{d_1}e_1,\quad
 v_2=x_1\sqrt{d_1}e_1+\sqrt{d_2}e_2,\quad
 v_3=x_2\sqrt{d_1}e_1+x_3\sqrt{d_2}e_2+\sqrt{d_3}e_3.
\]
Add an integral multiple of $v_1$ to $v_2$ to arrange
$|x_1|\le1/2$. Add an integral multiple of $v_2$ to $v_3$ to arrange
$|x_3|\le1/2$, and then an integral multiple of $v_1$ to arrange
$|x_2|\le1/2$. These changes preserve the basis and the shortest
projected vector $p(v_2)$.
Changing the signs of the basis vectors sends the shears to
$(s_1s_2x_1,s_1s_3x_2,s_2s_3x_3)$, with $s_i\in\{\pm1\}$.
Take $s_2=1$ and choose $s_1,s_3$ to make $x_1,x_3$ nonnegative.
If the resulting $x_2$ is negative, replace $v_3$ by $v_3+v_1$.
Then $0\le x_2\le1$, with $x_1,x_3$ unchanged.

The five nonzero vectors
$v_1+v_2-v_3$, $v_2-v_3$, $v_2$, $v_3$, and $v_3-v_1$ have squared
length at least $d_1$. Computing these lengths and dividing by $d_1$
gives the first five inequalities in \eqref{eq:reduction-inequalities}.
The shortest-vector property of $p(v_2)$ gives
$x_3^2d_2+d_3\ge d_2$, which is the sixth inequality.
Finally, $d_1d_2d_3=|L|^2=1$; the displayed formulas for $d_i$
follow from $y_1=d_1/d_2$ and $y_2=d_2/d_3$.
\end{proof}

\begin{lemma}[Uniform cusp exclusion]\label{lem:endpoint-cusp}
For every $c\ge0$, if $\lambda_1(L)\le4/5$, then
\[
 P(L,c)>P(\F,c).
\]
\end{lemma}
\begin{proof}
For $q\ge16/25$, we claim that
\begin{equation}\label{eq:cusp-kernel-ratios}
 \frac{I(q,c)}{I(16/25,c)}\le\frac{I(q,0)}{I(16/25,0)},\qquad
 \frac{J(q,c)}{I(16/25,c)}\le\frac{J(q,0)}{I(16/25,0)}.
\end{equation}
The functions $t\mapsto e^{-c/t}$ and
$t\mapsto e^{-\pi(q-16/25)t}$ are respectively increasing and
decreasing on $[1,\infty)$. Hence, for $t,u\ge1$,
\[
 (e^{-c/t}-e^{-c/u})
 (e^{-\pi(q-16/25)t}-e^{-\pi(q-16/25)u})\le0.
\]
Multiply by $(tu)^{-1/2}e^{-(16\pi/25)(t+u)}$, integrate in $t,u$,
and expand. The resulting inequality is
\[
 I(q,c)I(16/25,0)\le I(q,0)I(16/25,c),
\]
which proves the first ratio bound. For the second, use
$J(q,c)\le e^{-c}J(q,0)$ and
$I(16/25,c)\ge e^{-c}I(16/25,0)$.
This argument proves the second ratio bound for every $q>0$.

All nonzero FCC vectors have squared length at least $2^{1/3}>16/25$.
Apply the first ratio bound to the direct FCC sum and the second,
which holds for every $q>0$, to the reciprocal BCC sum. Summing gives
\[
 \frac{P(\F,c)}{I(16/25,c)}
 \le\frac{P(\F,0)}{I(16/25,0)}<1.
\]
The strict scalar inequality is proved in Lemma~\ref{lem:cusp-scalar}
of the appendix. A shortest opposite pair in $L$ contributes
$I(\lambda_1(L)^2,c)$ to $P(L,c)$. Since $I(q,c)$ decreases in $q$,
\[
 P(L,c)\ge I(\lambda_1(L)^2,c)\ge I(16/25,c)>P(\F,c),
\]
as required.
\end{proof}

\begin{proof}[Proof of Proposition~\ref{prop:compact-reduction}]
It suffices to minimize over lattices with energy at most $P(\F,c)$.
Lemma~\ref{lem:endpoint-cusp} implies that every such lattice has
$\lambda_1(L)>4/5$. Choose the basis in
Lemma~\ref{lem:elementary-reduction}. Since $0\le x_1,x_3\le1/2$,
the third and sixth inequalities in \eqref{eq:reduction-inequalities}
give
\[
 y_1^{-1}\ge1-x_1^2\ge\frac34,\qquad
 y_2^{-1}\ge1-x_3^2\ge\frac34.
\]
Thus $y_1,y_2\le4/3$. Also
$d_1^3=y_1^2y_2>(16/25)^3$. Hence
\[
 y_1^2>\frac34\left(\frac{16}{25}\right)^3
 =\frac{3072}{15625}>\left(\frac{11}{25}\right)^2,
 \qquad
 y_2>\left(\frac34\right)^2\left(\frac{16}{25}\right)^3
 =\frac{2304}{15625}>\frac{147}{1000}.
\]
Thus all relevant lattices have representatives in
\eqref{eq:endpoint-rational-root}. The corresponding closed rectangle
of logarithmic coordinates is compact, and its continuous Gram image
consists of positive definite determinant-one matrices. The Ewald sums
converge normally on this image, so $P(\cdot,c)$ is continuous there
and attains a minimum. The rectangle contains a representative of
$\F$, and every lattice outside its image either has energy larger
than $P(\F,c)$ or has another representative inside. Its minimum is
therefore the global minimum. The cusp comparison proves the asserted
strict bound on its shortest vector.
\end{proof}

\begin{lemma}[Separate matrix convexity]\label{lem:matrix-convexity}
Let $\xi=(\ell_1,\ell_2,x_1,x_2,x_3)$. Both $Y(\xi)$ and
$Y(\xi)^{-1}$ are convex in the Loewner order in each individual
coordinate, with the other four fixed. Consequently, for
$S,T\succeq0$ and a coordinate rectangle $\mathcal B$,
\begin{equation}\label{eq:corner-maximum}
 \tr\{SY(\xi)+TY(\xi)^{-1}\}
 \le\max_{\zeta\in\mathcal C(\mathcal B)}
       \tr\{SY(\zeta)+TY(\zeta)^{-1}\}
 \qquad(\xi\in\mathcal B).
\end{equation}
\end{lemma}
\begin{proof}
The matrix $Y=N^TDN$ is a sum of positive $d_i$ times the outer
products of the rows of $N$. Each row is affine in an individual shear
variable, so its outer product has a positive semidefinite second
derivative in that variable. Each $d_i$ is the exponential of an affine
function of $(\ell_1,\ell_2)$, with the outer product fixed.
The same argument applies to $Y^{-1}=N^{-1}D^{-1}N^{-T}$, using
\[
 N^{-1}=\begin{pmatrix}
 1&-x_1&x_1x_3-x_2\\0&1&-x_3\\0&0&1
 \end{pmatrix}.
\]
Each column of $N^{-1}$ is affine in each shear separately.
Taking the trace against positive semidefinite matrices preserves
convexity. Successive one-dimensional convex interpolation then proves
\eqref{eq:corner-maximum}. The same interpolation proves the
corresponding matrix inequalities in the Loewner order.
\end{proof}

\section{The local comparison near FCC and BCC}
\label{sec:contacts}\label{sec:certified-local}

We prove local comparisons on neighborhoods whose radii are independent
of the screening parameter. On bounded screening intervals we estimate
the quadratic and cubic variations and bound the fourth-order remainder
uniformly. For large screening, the twelve FCC contact vectors give
the required estimate directly. This is the local part of the method
developed in our previous work~\cite{LW}.

Throughout this section, $c=\kappa^2/(4\pi)$ and $F_0=W$.
For $\kappa>0$, replacing $F_\kappa$ by $E_\kappa$ leaves every
energy difference and shape derivative unchanged.

\Needspace{15\baselineskip}
\begin{theorem}\label{thm:profile-local}\label{prop:profile-high-local}
Let $X=X^T$, $\tr X=0$, and $\|X\|_{\fro}\le1/40$. Then
\begin{equation}\label{eq:profile-B-local}
 F_\kappa(e^{X/2}\B)-F_\kappa(\B)
 \ge\frac1{1000}\|X\|_{\fro}^2
 \qquad(0\le c\le1/4),
\end{equation}
and, for every $\kappa\ge0$,
\begin{equation}\label{eq:profile-F-local}
 F_\kappa(e^{X/2}\F)-F_\kappa(\F)
 \ge\|X\|_{\fro}^2
 \begin{cases}
 10^{-6},&c\le10,\\
 (\kappa a+1)e^{-\kappa a}/(200a),&c>10.
 \end{cases}
\end{equation}
Moreover, if $A=A^T>0$, $\det A=1$, and
$\|A-\Id\|_{\op}\le3/100$, then
\begin{equation}\label{eq:profile-high-local}
 E_\kappa(A^{1/2}\F)-E_\kappa(\F)
 \ge\frac{s+1}{160a}e^{-s}\|A-\Id\|_{\fro}^2,
 \qquad s=\kappa a\ge12.
\end{equation}
Equality of the two energies in any of these comparisons forces the
deformation to be the identity.
\end{theorem}

We first give the elementary estimates needed for
\eqref{eq:profile-high-local}.

\begin{lemma}\label{lem:det}
Suppose $A=\Id+T=A^T>0$, $\det A=1$, and
$\|T\|_{\op}\le d<2/3$. Write $T=\eta\Id+S$, where $\tr S=0$.
Then
\begin{equation}\label{eq:det}
 \frac{\|T\|_{\fro}^2}{6(1+d)}\le\eta
 \le\frac{\|T\|_{\fro}^2}{6(1-d)},\qquad
 \eta\le\frac{\|S\|_{\fro}^2}{6-9d},\qquad
 \|S\|_{\fro}\le\sqrt3d.
\end{equation}
\end{lemma}
\begin{proof}
For $|t|\le d$,
\[
 \frac{t^2}{2(1+d)}\le t-\log(1+t)
 \le\frac{t^2}{2(1-d)}.
\]
Sum over the eigenvalues of $T$, using $\log\det A=0$, to get
\begin{equation}\label{eq:trace}
 \frac{\|T\|_{\fro}^2}{2(1+d)}\le\tr T
 \le\frac{\|T\|_{\fro}^2}{2(1-d)}.
\end{equation}
Thus $0\le\eta\le d$. Since $\tr T=3\eta$ and
$\|T\|_{\fro}^2=\|S\|_{\fro}^2+3\eta^2$, the upper bound gives
\[
 6(1-d)\eta\le\|S\|_{\fro}^2+3\eta^2
 \le\|S\|_{\fro}^2+3d\eta.
\]
Rearranging yields $(6-9d)\eta\le\|S\|_{\fro}^2$.
Finally, $\|S\|_{\fro}\le\|T\|_{\fro}\le\sqrt3d$.
\end{proof}

\begin{lemma}\label{lem:moment}
If $x_1,\ldots,x_6\in\R$, $\sum_jx_j=0$, and
$\sigma^2=\frac16\sum_jx_j^2$, then
\begin{equation}\label{eq:moment}
 \log\left(\frac16\sum_{j=1}^6e^{x_j}\right)
 \ge\frac1{13}\min\{\sigma^2,\sigma\}.
\end{equation}
\end{lemma}
\begin{proof}
There is nothing to prove when $\sigma=0$. Since the entries sum to
zero and at most five are positive,
\[
 \sum_{x_j<0}x_j^2\le\left(\sum_{x_j>0}x_j\right)^2
 \le5\sum_{x_j>0}x_j^2,
\]
and hence $\sum_{x_j>0}x_j^2\ge\sigma^2$.
Apply $e^x\ge1+x$ to negative entries and
$e^x\ge1+x+x^2/2$ to positive entries. This yields
$\frac16\sum e^{x_j}\ge1+\sigma^2/12$, so its logarithm is at least
$\sigma^2/(12+\sigma^2)$. This proves the assertion for
$0\le\sigma\le8$.

If $\sigma\ge8$, put $m=\max_jx_j$. Every entry lies in $[-5m,m]$.
Averaging $(m-x_j)(x_j+5m)\ge0$ gives $\sigma^2\le5m^2$; therefore
\[
 \log\left(\frac16\sum e^{x_j}\right)
 \ge\frac\sigma{\sqrt5}-\log6
 >\frac{4\sigma}{9}-2\ge\frac\sigma{13}.
\]
\end{proof}

\begin{proposition}\label{prop:tail}
Let $N_n$ denote the number of FCC vectors of squared length $na^2$.
Then $N_1=12$, $N_2=6$, $N_n<30n$ for every $n\ge1$, and
\begin{equation}\label{eq:tail}
 \sum_{n\ge2}N_ne^{-s(\sqrt n-1)}<3,
 \qquad E_\kappa(\F)<\frac{15}{2a}e^{-s}
 \qquad(s=\kappa a\ge10).
\end{equation}
\end{proposition}
\begin{proof}
The $n$th shell is represented by integer triples of squared norm $2n$
and even coordinate sum. The first two counts are immediate. Fixing
the first two coordinates leaves at most two choices for the third, so
\[
 N_n\le2(2\lfloor\sqrt{2n}\rfloor+1)^2
 \le16n+8\sqrt{2n}+2<30n.
\]
For $s\ge10$, the function $x e^{-s(\sqrt x-1)}$ decreases on
$[2,\infty)$. Sum--integral comparison gives
\begin{align*}
 \sum_{n\ge2}N_ne^{-s(\sqrt n-1)}
 &<30e^{-s(\sqrt2-1)}
 \left[2+2\left(\frac{2\sqrt2}s+\frac6{s^2}
                  +\frac{6\sqrt2}{s^3}+\frac6{s^4}\right)\right]\\
 &<90e^{-4}<3.
\end{align*}
Here the bracket is less than
$2+6/10+12/10^2+18/10^3+12/10^4<3$, and
$\sqrt2-1>2/5$. Summing the energy by shells now proves its stated
upper bound.
\end{proof}

For the bounded parameter range, consider the determinant-preserving
deformations $e^{X/2}C$, where $C\in\{\F,\B\}$, $X=X^T$, and
$\tr X=0$.

For $A>0$ and $j\in\mathbb Z$, define
\[
 K_j(A,c)=\int_1^\infty t^{j-1/2}e^{-At-c/t}\,dt,\qquad
 T_j(A,c)=\int_1^\infty t^j e^{-(A+c)t}\,dt,
\]
and write $K_j(A)=K_j(A,0)$.
These kernels satisfy $K_0(\pi q,c)=I(q,c)$ and
$T_0(\pi q,c)=J(q,c)$, so \eqref{eq:positive-P} gives the
derivatives of $X\mapsto F_\kappa(e^{X/2}C)$. Normal convergence, as
established in Section~\ref{sec:representation}, justifies the
following differentiations, also at $c=0$.

For a cubic shell $S$ set
\[
 M(S)=\sum_{v\in S}v_1^2,\qquad
 A_D(S)=\sum_{v\in S}(v_1^4-v_1^2v_2^2),\qquad
 A_O(S)=2\sum_{v\in S}v_1^2v_2^2.
\]
For a traceless symmetric matrix $H$ put
\[
 I_1(H)=\sum_iH_{ii}^3,\qquad
 I_2(H)=H_{11}H_{23}^2+H_{22}H_{13}^2+H_{33}H_{12}^2,\qquad
 I_3(H)=H_{12}H_{13}H_{23}.
\]
\Needspace{18\baselineskip}
\begin{lemma}[Variations at the cubic lattices]\label{lem:local-variation-formulas}
The first variation at $C\in\{\F,\B\}$ vanishes on traceless symmetric
directions. Its Hessian is
scalar on the diagonal traceless and off-diagonal subspaces, with
Frobenius eigenvalues
\begin{align}\label{eq:profile-mu}
 \mu_\tau(c)={}&\frac12\sum_{S\subset C}
 \{\pi^2 A_\tau(S)K_2(\pi q_S,c)-\pi M(S)K_1(\pi q_S,c)\}\notag\\
 &+\frac12\sum_{S\subset C^*}
 \{\pi^2 A_\tau(S)T_2(\pi q_S,c)-\pi M(S)T_1(\pi q_S,c)\},
 \qquad \tau\in\{D,O\},
\end{align}
where $q_S=|v|^2$ on the shell $S$. The third variation satisfies
\begin{equation}\label{eq:profile-cubic}
 \left.\frac{d^3}{dt^3}F_\kappa(e^{tH/2}C)\right|_{t=0}=\alpha(c)I_1(H)+\beta(c)I_2(H)+\gamma(c)I_3(H).
\end{equation}
Here $H=H^T$ and $\tr H=0$.
For $\|H\|_{\fro}=1$,
\begin{equation}\label{eq:profile-cubicnorm}
 |I_1(H)|\le\frac1{\sqrt6}<\frac5{12},\qquad
 |I_2(H)|\le\frac{\sqrt2}{9}<\frac16,\qquad
 |I_3(H)|\le\frac1{6\sqrt6}<\frac5{72}.
\end{equation}
\end{lemma}
\begin{proof}
Sign changes and permutations make the first derivative zero and split
the Hessian into the stated eigenspaces. The first two derivatives of
$v^Te^{tH}v$ at zero are $v^THv$ and $v^TH^2v$. For a unit Frobenius
direction in the diagonal traceless or off-diagonal subspace, shell
symmetry gives, respectively,
\[
 \sum_{v\in S}(v^THv)^2=A_\tau(S),\qquad
 \sum_{v\in S}v^TH^2v=M(S)\qquad(\tau=D,O).
\]
Differentiation under the integral gives
$I_q(q,c)=-\pi K_1(\pi q,c)$ and
$I_{qq}(q,c)=\pi^2K_2(\pi q,c)$. The chain rule now gives the direct
part of \eqref{eq:profile-mu}; the dual deformation $e^{-tH/2}$ gives
the same second-variation formula with $T_j$ in place of $K_j$.
The fourth moments are nonnegative, since
$A_D(S)=\tfrac12\sum_S(v_1^2-v_2^2)^2$ and $A_O(S)\ge0$.

A cubic invariant under coordinate sign changes and permutations is a
linear combination of $I_1,I_2,I_3$. A surviving monomial contains
three diagonal entries, one diagonal entry and a squared off-diagonal
entry, or the three different off-diagonal entries; $\tr H=0$ gives
the stated reduction. For the middle bound in
\eqref{eq:profile-cubicnorm}, put $d^2=\sum H_{ii}^2$. Since
$\max_i|H_{ii}|\le\sqrt{2/3}d$ and
$\sum_{i<j}H_{ij}^2=(1-d^2)/2$,
\[
 |I_2(H)|\le\max_i|H_{ii}|\sum_{i<j}H_{ij}^2
 \le\frac{\sqrt{2/3}}2d(1-d^2)\le\frac{\sqrt2}{9}.
\]
The last step uses $\max_{0\le d\le1}d(1-d^2)=2/(3\sqrt3)$.
The other two bounds follow from the traceless diagonal constraint and
the arithmetic--geometric mean inequality.

For the coefficients in \eqref{eq:profile-cubic}, use the matrices
\[
 H_D=\begin{pmatrix}1&0&0\\0&1&0\\0&0&-2\end{pmatrix},\quad
 H_T=\begin{pmatrix}1&0&0\\0&-1&1\\0&1&0\end{pmatrix},\quad
 H_O=\begin{pmatrix}0&1&1\\1&0&1\\1&1&0\end{pmatrix}.
\]
Their cubic invariants are $(-6,0,0)$, $(0,1,0)$, and $(0,0,1)$.
For a shell represented by integer vectors $m$ with physical vector
$v=\sqrt{s}\,m$, define
\[
 U_3(H)=\sum_m(m^THm)^3,\quad
 U_{12}(H)=\sum_m(m^THm)(m^TH^2m),\quad
 U_1(H)=\sum_m m^TH^3m.
\]
Its direct contribution to the directional third derivative is
\begin{equation}\label{eq:profile-thirdmoment}
 \frac12\{-\pi^3s^3U_3(H)K_3
       +3\pi^2s^2U_{12}(H)K_2-\pi sU_1(H)K_1\}.
\end{equation}
The arguments are $(\pi s|m|^2,c)$. The dual contribution is the
negative of the same expression with $T_j$ in place of $K_j$, because
the dual deformation is $e^{-X/2}$. Sum these expressions and evaluate
on $H_D,H_T,H_O$ to obtain $-6\alpha,\beta,\gamma$.
\end{proof}

\begin{lemma}[Uniform derivative tails]\label{lem:local-derivative-tails}
Let $\delta=1/40$. After retaining FCC shells $n\le9$ and BCC shells
$\ell\le30$, the omitted directional second and third variations at
the origin are less in absolute value than $10^{-10}$ and $10^{-8}$,
respectively, for $\|H\|_{\fro}=1$. The omitted fourth variation
on $\|X\|_{\fro}\le\delta$ is less than $10^{-6}$ under the same
normalization. These bounds hold for both references and every $c\ge0$.
\end{lemma}
\begin{proof}
Set $\delta=1/40$. If $\|X\|_{\fro}\le\delta$ and
$\|H\|_{\fro}=1$, then all derivatives of
$v^Te^{X+uH}v$ of order $j\le4$ at $u=0$ have absolute value at most
$e^\delta|v|^2$. This follows either by the iterated integral formula
for derivatives of the matrix exponential or, along the radial Taylor
path, by diagonalizing $H$. Also
$v^Te^Xv\ge e^{-\delta}|v|^2$.
Put $A=\pi e^{-\delta}|v|^2$ and $U=\pi e^\delta|v|^2$.
The fourth derivative of a summand is therefore bounded, after
integration, by
\begin{equation}\label{eq:profile-Bell4}
 U^4K_4(A,c)+6U^3K_3(A,c)+7U^2K_2(A,c)+UK_1(A,c),
\end{equation}
and the same formula with $T_j$ for a dual summand. The coefficients
$1,6,7,1$ come from the fourth derivative of an exponential of a
function. The factor $1/2$ from \eqref{eq:positive-P} is retained when
summing this estimate.

We spell out all omitted-shell estimates. FCC shells have
$|v|^2=2^{1/3}n$, with $N_n<30n$. BCC can be written
\[
 \B=\left\{2^{-2/3}m:m\in\Z^3,
                 \ m_1\equiv m_2\equiv m_3\pmod2\right\}.
\]
The number of integer triples of squared norm $\ell$ is less than
$30\ell$, by fixing two coordinates as in Proposition~\ref{prop:tail}.
For every omitted vector, $A=\pi e^{-\delta}|v|^2>3$, and
\begin{equation}\label{eq:profile-ratios}
 \begin{alignedat}{3}
 2^{1/3}\pi&<4,&\qquad 2^{-4/3}\pi&<\frac54,&\qquad
 e^{8\delta}&<\frac54,\\
 e^{-39\cdot2^{1/3}\pi/40}&<\frac1{47},&\qquad
 e^{-39\cdot2^{-4/3}\pi/40}&<\frac3{10}.
 \end{alignedat}
\end{equation}
Since $e^{-\delta}>39/40$, the two exponential bounds also hold
with $39/40$ replaced by $e^{-\delta}$.
Dropping screening and using $t^{-1/2}\le1$ shows that an omitted
summand in \eqref{eq:profile-Bell4}, including its factor $1/2$, is at
most
\[
 \frac58 e^{-A}(A^3+10A^2+37A+75+75/A)
 \le\frac58e^{-A}(A^3+10A^2+37A+100).
\]
Thus, after retaining FCC shells $n\le9$ and BCC shells $\ell\le30$,
the omitted fourth derivative norm is less than
\begin{align}\label{eq:profile-tail4}
 R_4={}&\frac{75}{4}\sum_{n=10}^\infty
 \frac{n(64n^3+160n^2+148n+100)}{47^n}\notag\\
 &+\frac{75}{4}\sum_{\ell=31}^\infty
 \ell\left(\frac3{10}\right)^\ell
 \left(\frac{125}{64}\ell^3+\frac{125}{8}\ell^2
              +\frac{185}{4}\ell+100\right)
 <10^{-6}.
\end{align}
The fourth derivative bound $M_4$ used below is the finite sum in
\eqref{eq:profile-Bell4}, plus $10^{-6}$.

At the origin the corresponding absolute Hessian tail is bounded by
\begin{equation}\label{eq:profile-tail2}
 R_2=30\sum_{n\ge10}\frac{n(4n+20)}{49^n}
 +30\sum_{\ell\ge31}\ell(3/10)^\ell(5\ell/4+20)
 <10^{-10}.
\end{equation}
Indeed, a unit Frobenius direction has $|v^THv|\le|v|^2$ and
$|v^TH^2v|\le|v|^2$. After dropping screening, the contribution of a shell $S$
in absolute value is at most
$\tfrac12\#S\,e^{-A}(A+3+3/A)\le\#S\,e^{-A}(A+20)$, where
$A=\pi|v|^2>3$; this proves the displayed tail bound.
For completeness, integrating the third derivative polynomial
$(At)^3+3(At)^2+At$ bounds the omitted third derivative norm by
\begin{equation}\label{eq:profile-tail3}
 R_3=18\sum_{n\ge10}\frac{n(16n^2+24n+26)}{49^n}
 +18\sum_{\ell\ge31}\ell(3/10)^\ell(25\ell^2/16+15\ell/2+26)<10^{-8}.
\end{equation}
For example, the polynomial before its final simplification is
$e^{-A}(A^2+6A+13+13/A)$, and the prefactor $3/5$ used here exceeds
the required $1/2$.
All inequalities in \eqref{eq:profile-tail4}--\eqref{eq:profile-tail3}
are rational: for their evaluation one can use
\[
 \sum_{n=N}^\infty n^k r^n
 =\left(r\frac{d}{dr}\right)^k\frac{r^N}{1-r},\qquad 0<r<1.
\]
\end{proof}

For the truncated FCC sum, retain the integer vectors satisfying
\[
 z\in\Z^3,\qquad 0<|z|^2\le18,\qquad z_1+z_2+z_3\in2\Z;
\]
for the BCC sum, retain those satisfying
\[
 z\in\Z^3,\qquad 0<|z|^2\le30,\qquad
 z_1\equiv z_2\equiv z_3\pmod2.
\]
Group the vectors by $m=|z|^2$. Every coordinate lies in $[-6,6]$.
The physical squared scales are $2^{-2/3}$ for FCC and $2^{-4/3}$ for
BCC; $s$ denotes the applicable scale below. For $C=\F$ the direct sum
uses the first family and the dual sum the second; for $C=\B$
interchange the two families. For each shell form
the integer moments
\[
 m_2=\sum_z z_1^2,\quad
 a_D=\sum_z(z_1^4-z_1^2z_2^2),\quad
 a_O=2\sum_z z_1^2z_2^2,
\]
and the three sets of integer moments in
\eqref{eq:profile-thirdmoment}. The physical moments in
\eqref{eq:profile-mu} are $M=sm_2$ and $A_\tau=s^2a_\tau$.
Thus every retained shell moment is a finite integer sum times a power
of $2^{-2/3}$ or $2^{-4/3}$.

For $1\le j\le4$ put
\[
 L_j(A,c)=\sum_{n=0}^{61}\frac{(-c)^n}{n!}K_{j-n}(A),\qquad
 U_j(A,c)=\sum_{n=0}^{60}\frac{(-c)^n}{n!}K_{j-n}(A).
\]
Only the indices down to $-60$ are needed in the expressions below.
Taylor's formula with its signed integral remainder gives, for every
$c\ge0$,
\begin{equation}\label{eq:profile-Ksandwich}
 L_j(A,c)\le K_j(A,c)\le U_j(A,c).
\end{equation}
No assumption about decreasing Taylor terms is needed. Both $K_j(A,c)$
and $T_j(A,c)$ decrease in $c$.

On a closed parameter interval $[u,v]$, define $h_\tau$ by the finite sum
\begin{align}\label{eq:profile-hlower}
 h_\tau={}&\frac12\sum_{\text{direct shells}}
 \{\pi^2s^2a_\tau L_2(\pi sm,v)
                  -\pi sm_2 U_1(\pi sm,u)\}\notag\\
 &+\frac12\sum_{\text{dual shells}}
 \{\pi^2s^2a_\tau T_2(\pi sm,v)
                  -\pi sm_2 T_1(\pi sm,u)\}-10^{-10}.
\end{align}
Thus $\mu_\tau(c)\ge h_\tau$ throughout this interval. Set
$h=\min(h_D,h_O)$.
To bound the cubic variation, enclose the three cubic coefficients using
\eqref{eq:profile-thirdmoment} and, in a direct summand,
\[
 K_j(\pi sm,c)\in[L_j(\pi sm,v),U_j(\pi sm,u)],
\]
and in a dual summand use
$[T_j(\pi sm,v),T_j(\pi sm,u)]$.
Summing the signed moment enclosures and adding
$[-10^{-6},10^{-6}]$ encloses each full cubic coefficient.
The added interval covers the omitted shells: each test matrix has
Frobenius norm cubed less than $15$, and $15R_3<10^{-6}$.
Let $b_\alpha,b_\beta,b_\gamma$ be the corresponding absolute upper
bounds, and put
\begin{equation}\label{eq:profile-Cbound}
 C_3=\frac5{12}b_\alpha+\frac16b_\beta+\frac5{72}b_\gamma.
\end{equation}
Finally, define $M_4$ by the finite fourth derivative sum
\eqref{eq:profile-Bell4}, taking $c=u$, replacing each direct $K_j$ by
$U_j$, retaining each dual $T_j$ exactly, and adding $10^{-6}$.
The resulting bound holds on the whole shape ball and parameter interval.

\begin{lemma}\label{lem:local-taylor-profile}
For $C=\F$, let $0\le j\le999$; for $C=\B$, let $0\le j\le24$.
There are bounds $h$, $C_3$, and $M_4$, uniform for
$c\in[j/100,(j+1)/100]$, such that, for every symmetric traceless
matrix $H$,
\begin{align*}
 \left.\frac d{dt}F_\kappa(e^{tH/2}C)\right|_{t=0}&=0,\\
 \left.\frac{d^2}{dt^2}F_\kappa(e^{tH/2}C)\right|_{t=0}
 &\ge h\|H\|_{\fro}^2,\\
 \left|\frac{d^3}{dt^3}F_\kappa(e^{tH/2}C)\right|_{t=0}
 &\le C_3\|H\|_{\fro}^3,\\
 \left|\frac{d^4}{dt^4}F_\kappa(e^{(X+tH)/2}C)\right|_{t=0}
 &\le M_4\|H\|_{\fro}^4
 \qquad(\|X\|_{\fro}\le1/40).
\end{align*}
They satisfy
\begin{equation}\label{eq:local-taylor-margin}
 \frac h2-\frac{C_3}{240}-\frac{M_4}{38400}
 >\begin{cases}10^{-6},&C=\F,\\10^{-3},&C=\B.\end{cases}
\end{equation}
\end{lemma}
\begin{proof}
The preceding derivative formulas and truncation bounds give the four
variation estimates, uniformly on each closed parameter interval and
on the whole shape ball. The finite endpoint inequalities in
Lemma~\ref{lem:local-profile-checks} of Appendix~\ref{app:local-numerics}
are precisely \eqref{eq:local-taylor-margin}.
\end{proof}

\begin{proof}[Proof of Theorem~\ref{thm:profile-local}]
Taylor's theorem and Lemma~\ref{lem:local-taylor-profile} give
\begin{align*}
 F_\kappa(e^{X/2}C)-F_\kappa(C)
 &\ge\frac h2\|X\|_{\fro}^2
       -\frac{C_3}{6}\|X\|_{\fro}^3
       -\frac{M_4}{24}\|X\|_{\fro}^4\\
 &\ge\left(\frac h2-\frac{C_3}{240}
                    -\frac{M_4}{38400}\right)\|X\|_{\fro}^2.
\end{align*}
This proves \eqref{eq:profile-B-local} and the part $c\le10$ of
\eqref{eq:profile-F-local}.

We next prove \eqref{eq:profile-high-local}. Put $d=3/100$,
$s=\kappa a$, and write
$A-\Id=T=\eta\Id+S$ as in Lemma~\ref{lem:det}. In particular,
\begin{equation}\label{eq:profile-high-det}
 \eta\le\frac9{50}\|S\|_{\fro}^2,\qquad
 \|S\|_{\fro}<\frac{13}{250}.
\end{equation}
Set $f(q)=q^{-1/2}e^{-\kappa\sqrt q}$ and $\beta=(s+1)/2$.
The identity
$(\log f)''(q)=\kappa/(4q^{3/2})+1/(2q^2)>0$ gives
\begin{equation}\label{eq:logtangent}
 f(a^2(1+b))\ge f(a^2)e^{-\beta b}\qquad(b>-1).
\end{equation}

We first bound the possible loss from all higher shells. On a shell of
squared radius $q=a^2n$, sign changes and permutations imply
$\sum vv^T=(N_nq/3)\Id$. Since $f$ is convex,
\[
 \sum_{v\text{ in shell}}\{f(v^TAv)-f(q)\}
 \ge f'(q)\sum_vv^TTv=\eta N_nqf'(q).
\]
Moreover,
\[
 -qf'(q)=\frac{\kappa\sqrt q+1}{2}f(q)
 \le\beta f(a^2)e^{-s(\sqrt n-1)},
\]
because $(s\sqrt n+1)/\sqrt n\le s+1$.
Proposition~\ref{prop:tail} therefore bounds the total higher-shell
loss, before the pair factor, by $3\beta\eta f(a^2)$.

For the twelve unit contact directions $w=v/a$,
\[
 |w^TTw|\le d,\qquad \sum_w w^TTw=12\eta,
\]
\[
 \sum_w(w^TTw)^2
 =\sum_iT_{ii}^2+(\tr T)^2+4\sum_{i<j}T_{ij}^2
 \ge\|T\|_{\fro}^2.
\]
Using $e^{-\beta b}\ge1-\beta b+\frac12\beta^2e^{-\beta d}b^2$
in \eqref{eq:logtangent} gives the first-shell increase before the pair factor:
\[
 \sum_w\{f(a^2(1+w^TTw))-f(a^2)\}
 \ge f(a^2)\left(\frac{\beta^2}{2}e^{-\beta d}\|T\|_{\fro}^2-12\beta\eta\right).
\]
The higher shells can decrease the energy by at most
$3\beta\eta f(a^2)$. By \eqref{eq:det},
\[
 15\eta\le\frac{15}{6(1-d)}\|T\|_{\fro}^2
 =\frac{250}{97}\|T\|_{\fro}^2.
\]
Thus the unhalved energy increase is bounded below by
\begin{equation}\label{eq:profile-high-quadratic}
 \beta f(a^2)\left(\frac\beta2e^{-3\beta/100}
                   -\frac{250}{97}\right)\|T\|_{\fro}^2.
\end{equation}
For $13/2\le\beta\le64$, the function
$\beta e^{-3\beta/100}$ attains its minimum at an endpoint.
Its value at the left endpoint is at least
$\frac{13}{2}(1-39/200)=2093/400$, while at the
right endpoint it exceeds $8$, since $e^{48/25}<8$. Thus
\[
 \frac\beta2e^{-3\beta/100}-\frac{250}{97}
 \ge\frac{2093}{800}-\frac{250}{97}
 =\frac{3021}{77600}>\frac1{40}.
\]
Restoring the pair factor proves \eqref{eq:profile-high-local}
for this range of $\beta$.

For $\beta\ge64$, the twelve contact values are
$\eta-S_{kk}/2\pm S_{ij}$, each repeated twice, with
$(k,i,j)=(1,2,3),(2,1,3),(3,1,2)$. Hence their deformed energy is
at least $12f(a^2)e^{-\beta\eta}M$, where
\[
 M=\frac16\sum_{(k,i,j)}
       \bigl(e^{\beta S_{kk}/2+\beta S_{ij}}
             +e^{\beta S_{kk}/2-\beta S_{ij}}\bigr).
\]
The six exponents have mean zero and variance
\begin{equation}\label{eq:variance}
 \sigma^2=\frac{\beta^2}{12}
 \left(\sum_iS_{ii}^2+4\sum_{i<j}S_{ij}^2\right)
 \ge\frac{\beta^2}{12}\|S\|_{\fro}^2.
\end{equation}
Lemma~\ref{lem:moment} and \eqref{eq:variance} give
\[
 \begin{aligned}
 \sigma\le1:\quad \log M
 &\ge\frac{\sigma^2}{13}
 \ge\frac{\beta^2}{156}\|S\|_{\fro}^2
 \ge\frac9{25}\beta\|S\|_{\fro}^2\ge2\beta\eta,\\
 \sigma\ge1:\quad \log M
 &\ge\frac\sigma{13}
 \ge\frac\beta{13\sqrt{12}}\|S\|_{\fro}
 \ge\frac9{25}\beta\|S\|_{\fro}^2\ge2\beta\eta.
 \end{aligned}
\]
Here the first line uses $\beta\ge64$ and $64/156>9/25$;
the second uses $\|S\|_{\fro}<13/250$ and
$1/(13\sqrt{12})>1/52>(9/25)(13/250)$.
The final inequalities are \eqref{eq:profile-high-det}.
The first-shell gain is therefore at least
$12f(a^2)(e^{\beta\eta}-1)\ge12\beta\eta f(a^2)$.
After the higher-shell loss and the pair factor,
\[
 E_\kappa(A^{1/2}\F)-E_\kappa(\F)
 \ge\frac92\beta\eta f(a^2)
 \ge\frac9{12(1+d)}\beta f(a^2)\|T\|_{\fro}^2,
\]
which is stronger than \eqref{eq:profile-high-local}. Absolute
convergence justifies summing the shellwise tangent inequalities.

It remains to prove \eqref{eq:profile-F-local} when $c>10$.
Then $s^2=2^{7/3}\pi c>156>144$. For traceless symmetric $X$,
$\|X\|_{\op}\le\sqrt{2/3}\|X\|_{\fro}<1/48$, so
\[
 \|e^X-\Id\|_{\op}<e^{1/48}-1<1/47<3/100.
\]
Comparison of the exponential and geometric series proves the middle
inequality. Also,
$\|e^X-\Id\|_{\fro}^2\ge e^{-1/20}\|X\|_{\fro}^2
\ge(4/5)\|X\|_{\fro}^2$, by applying the scalar mean value theorem to
each eigenvalue. Thus \eqref{eq:profile-high-local}, with $A=e^X$, gives
\[
 \begin{aligned}
 E_\kappa(e^{X/2}\F)-E_\kappa(\F)
 &\ge\frac{s+1}{160a}e^{-s}\|e^X-\Id\|_{\fro}^2\\
 &\ge\frac{s+1}{200a}e^{-s}\|X\|_{\fro}^2.
 \end{aligned}
\]
Since $s=\kappa a$, this is the remaining part of
\eqref{eq:profile-F-local}. All coefficients are positive, which also
proves the asserted equality cases.
\end{proof}

\begin{corollary}\label{cor:profile-op-embedding}
The comparisons \eqref{eq:profile-B-local} and
\eqref{eq:profile-F-local} apply when $A=A^T>0$, $\det A=1$, and
$\|A-\Id\|_{\op}\le7/500$, with deformation $A^{1/2}$ of the
corresponding reference lattice and the same parameter ranges.
\end{corollary}
\begin{proof}
If $\det A=1$ and $\|A-\Id\|_{\op}\le7/500$, then
$X=\log A$ is symmetric and traceless, and
\[
 \|X\|_{\fro}^2\le3(7/493)^2
 =147/243049<1/1600.
\]
Here $|\log(1+t)|\le|t|/(1-7/500)$ on the stated interval.
\end{proof}

\section{Global comparison}\label{sec:global-comparison}
\label{sec:screening-profile-lemma}
We now compare all lattice shapes uniformly in the screening parameter.
The following theorem is the global estimate used to prove
Theorem~\ref{thm:main}.

\begin{theorem}\label{thm:global-comparison}
For every $c\ge0$ and every lattice $L$ with $|L|=1$,
\[
 P(L,c)\ge\min\{P(\F,c),P(\B,c)\}.
\]
Equality is possible only if $L$ is orthogonally equivalent to $\F$ or $\B$.
\end{theorem}

The cusp estimate and the elementary reduction restrict the comparison
to a fixed rectangle in five lattice coordinates. For $0\le c\le32$,
we subtract a cubic reference before expanding in the screening parameter;
this retains the cancellation in the energy difference. For
$\kappa\ge20$, convexity extends a comparison at one screening parameter
to the entire half-line. These two parameter ranges overlap.

Let $\mathcal Q$ be a rectangle in the reduced coordinates
$\xi=(\ell_1,\ell_2,x_1,x_2,x_3)$, and denote its $32$ corners by
$\mathcal C(\mathcal Q)$. Choose $\xi_0\in\mathcal Q$ and put
$Q=Y(\xi_0)$ and $Y_z=Y(z)$. For finite symmetric sets
$\mathcal M_I,\mathcal M_J\subset\Z^3\setminus\{0\}$, define
\begin{align*}
 p(c)&=\frac12\sum_{m\in\mathcal M_I}I(Q[m],c)
       +\frac12\sum_{m\in\mathcal M_J}J(Q^{-1}[m],c),\\
 S(c)&=-\frac12\sum_{m\in\mathcal M_I}I_q(Q[m],c)mm^T,\qquad
 T(c)=-\frac12\sum_{m\in\mathcal M_J}J_q(Q^{-1}[m],c)mm^T.
\end{align*}
Both $S(c)$ and $T(c)$ are positive semidefinite. Let $R$ be a reference
lattice, and let $p_R(c)$ be a finite positive Ewald sum satisfying
\begin{equation}\label{eq:profile-reference-enclosure}
 0\le P(R,c)-p_R(c)\le\tau_R
 \qquad(u\le c\le u+C),
\end{equation}
where $u\ge0$ and $C>0$. At each corner set
\[
 F_{R,z}(c)=p(c)-p_R(c)
 +\tr\{S(c)(Q-Y_z)+T(c)(Q^{-1}-Y_z^{-1})\}.
\]

\begin{lemma}[Comparison of screening profiles]\label{lem:screening-profile}
Suppose that, for $0\le s\le C$ and every $z\in\mathcal C(\mathcal Q)$,
\[
 \left|F_{R,z}(u+s)-\sum_{n=0}^Na_{z,n}s^n\right|
 \le r_zs^{N+1}.
\]
Define
\[
 b_{z,k}=\sum_{n=0}^k
 \frac{\binom{k}{n}}{\binom{N}{n}}C^na_{z,n},
 \qquad 0\le k\le N.
\]
Then, for $\xi\in\mathcal Q$ and $u\le c\le u+C$,
\begin{equation}\label{eq:screening-profile-bound}
 P(Y(\xi),c)-P(R,c)
 \ge\min_{\substack{z\in\mathcal C(\mathcal Q)\\0\le k\le N}}
       \{b_{z,k}-C^{N+1}r_z\}-\tau_R.
\end{equation}
When $N=0$, the coefficient formula means $b_{z,0}=a_{z,0}$.
\end{lemma}
\begin{proof}
Convexity of the kernels in squared distance and positivity of the
omitted competitor terms give
\[
 P(Y,c)\ge p(c)+\tr\{S(c)(Q-Y)+T(c)(Q^{-1}-Y^{-1})\}.
\]
By Lemma~\ref{lem:matrix-convexity},
$\tr\{S(c)Y+T(c)Y^{-1}\}$ is bounded above by its largest corner
value. Hence \eqref{eq:profile-reference-enclosure} implies
\[
 P(Y,c)-P(R,c)\ge\min_z F_{R,z}(c)-\tau_R.
\]
For $0\le n\le N$, the binomial theorem gives
\[
 t^n=\sum_{k=n}^N
 \frac{\binom{k}{n}}{\binom{N}{n}}\binom Nk t^k(1-t)^{N-k}.
\]
Consequently, with $t=s/C\in[0,1]$, the polynomial at corner $z$ has
the Bernstein representation
\[
 \sum_{n=0}^Na_{z,n}s^n
 =\sum_{k=0}^Nb_{z,k}\binom Nk t^k(1-t)^{N-k}.
\]
The basis functions are nonnegative and sum to one. The remainder
bound therefore gives \eqref{eq:screening-profile-bound}.
\end{proof}

The Taylor expansion is applied to the difference $F_{R,z}$, rather
than to separate upper and lower bounds for the two energies. Thus
cancellation between competitor and reference derivatives is retained.
The next lemma gives the coefficients and a uniform remainder on any
finite screening interval.

\begin{lemma}[Coefficients and remainders for the screening comparison]
\label{lem:global-profile-numerics}
For $q,\alpha>0$, $u\ge0$, and integers $j\ge-1$, $n\ge0$, set
\[
 H_j(q,u)=\int_1^\infty t^{-j-1/2}e^{-\pi qt-u/t}\,dt,
 \qquad
 A_n(\alpha)=\frac1{n!}\int_1^\infty t^ne^{-\alpha t}\,dt.
\]
The Taylor coefficients in Lemma~\ref{lem:screening-profile} are
\begin{align*}
 a_{z,n}={}&-\frac{p_R^{(n)}(u)}{n!}
 +\frac{(-1)^n}{2n!}\sum_{m\in\mathcal M_I}
 \Bigl[H_n(Q[m],u)
       +\pi H_{n-1}(Q[m],u)m^T(Q-Y_z)m\Bigr]\\
 &+\frac{(-1)^n}{2}\sum_{m\in\mathcal M_J}
 \left[\begin{aligned}[t]
 &A_n(\pi Q^{-1}[m]+u)\\
 &{}+\pi(n+1)A_{n+1}(\pi Q^{-1}[m]+u)
                      m^T(Q^{-1}-Y_z^{-1})m
 \end{aligned}\right].
\end{align*}
A remainder constant valid on every finite interval $0\le s\le C$ is
\begin{equation}\label{eq:profile-corner-remainder}
 \begin{aligned}
 r_z={}&\frac{(-1)^{N+1}p_R^{(N+1)}(u)}{(N+1)!}
 +\frac1{2(N+1)!}\sum_{m\in\mathcal M_I}H_{N+1}(Q[m],u)
 \\
 &+\frac12\sum_{m\in\mathcal M_J}A_{N+1}(\pi Q^{-1}[m]+u)
 \\
 &+\frac{\pi}{2(N+1)!}
 \sum_{i,j=1}^3\sum_{m\in\mathcal M_I}
 H_N(Q[m],u)\bigl|m_i m_j(Q-Y_z)_{ij}\bigr|
 \\
 &+\frac{\pi(N+2)}2
 \sum_{i,j=1}^3\sum_{m\in\mathcal M_J}
 A_{N+2}(\pi Q^{-1}[m]+u)
       \bigl|m_i m_j(Q^{-1}-Y_z^{-1})_{ij}\bigr|.
 \end{aligned}
\end{equation}
\end{lemma}
\begin{proof}
Differentiation under the defining integrals gives, in particular,
\begin{align*}
 -\frac{\partial_c^n I_q(q,u)}{n!}&=\frac{\pi(-1)^n}{n!}H_{n-1}(q,u),\\
 -\frac{\partial_c^n J_q(q,u)}{n!}
 &=\pi(-1)^n(n+1)A_{n+1}(\pi q+u).
\end{align*}
The factor $n+1$ in the reciprocal tangent term is the factorial ratio
$(n+1)!/n!$. The coefficient of $s^n$ in $e^{-sv}$ is $(-v)^n/n!$,
and, for $s,v\ge0$, its absolute remainder after degree $N$ is at most
$s^{N+1}v^{N+1}/(N+1)!$. Apply this with $v=1/t$ in the direct
kernel and $v=t$ in the reciprocal kernel, including their squared-length
derivatives. Substitution into $F_{R,z}$ gives the displayed
coefficients and the first three terms of the remainder bound.
The reference contribution is nonnegative because $p_R$ is completely
monotone in the screening variable.

For a symmetric matrix $D$ and nonnegative weights $w_m$,
\[
 \sum_m w_m|m^TDm|
 \le\sum_{i,j=1}^3\sum_m w_m|m_i m_jD_{ij}|.
\]
Apply this to $D=Q-Y_z$ and $D=Q^{-1}-Y_z^{-1}$, using the
corresponding derivative remainders as weights. This gives the last
two terms of \eqref{eq:profile-corner-remainder}. Absolute values are
taken before summing the rank-one contributions, so no cancellation of
matrix entries is used to bound a remainder.
\end{proof}

The moments can be evaluated by integration by parts. At $u=0$,
\begin{align*}
 H_0(q,0)&=\frac{\erfc(\sqrt{\pi q})}{\sqrt q},&
 H_{-1}(q,0)&=\frac{e^{-\pi q}+H_0(q,0)/2}{\pi q},\\
 H_n(q,0)&=\frac{e^{-\pi q}-\pi qH_{n-1}(q,0)}{n-1/2}
 \quad(n\ge1).
\end{align*}
For $u>0$, use $H_0(q,u)=I(q,u)$ and
\[
 H_1(q,u)=\frac{\sqrt\pi}{2\sqrt u}\Bigl[
 e^{-2\sqrt{\pi qu}}\erfc(\sqrt{\pi q}-\sqrt u)
 -e^{2\sqrt{\pi qu}}\erfc(\sqrt{\pi q}+\sqrt u)\Bigr].
\]
The remaining moments satisfy
\begin{align*}
 H_{-1}(q,u)&=
 \frac{e^{-\pi q-u}+H_0(q,u)/2+uH_1(q,u)}{\pi q},\\
 H_{j+1}(q,u)&=
 \frac{\pi qH_{j-1}(q,u)+(j-1/2)H_j(q,u)-e^{-\pi q-u}}u
 \qquad(j\ge1).
\end{align*}
For $\alpha>0$, the reciprocal moments satisfy
\[
 A_0(\alpha)=\frac{e^{-\alpha}}\alpha,\qquad
 A_n(\alpha)=\frac{A_{n-1}(\alpha)+e^{-\alpha}/n!}\alpha
 \quad(n\ge1).
\]

For the reference sum, the FCC squared lengths are $2^{1/3}n$, where
$N_n=\#\{m\in\Z^3:|m|^2=2n\}$, and the BCC squared lengths are
$2^{-4/3}\ell$, where $M_\ell$ counts integer triples of squared norm
$\ell$ whose coordinates have the same parity. We use the shells
$n\le48$ and $\ell\le192$ in the finite reference sum $p_R$.
They are completely enumerated in $[-14,14]^3$.
A uniform bound for the omitted reference shells is
\begin{equation}\label{eq:profile-reference-tail}
 \tau=\frac{15e^{-49\pi2^{1/3}}}{\pi2^{1/3}(1-e^{-\pi2^{1/3}})}
 +\frac{15e^{-193\pi2^{-4/3}}}{\pi2^{-4/3}(1-e^{-\pi2^{-4/3}})}.
\end{equation}
For $c\ge0$, both kernels satisfy
$I(q,c),J(q,c)\le e^{-\pi q}/(\pi q)$.
The shell bound $N_n<30n$ was proved in Proposition~\ref{prop:tail};
likewise, fixing two coordinates gives
$M_\ell\le2(2\sqrt\ell+1)^2\le18\ell<30\ell$.
Since $R$ and $R^*$ are one FCC and one BCC lattice, the factor
$1/2$ in each Ewald sum and geometric summation give
$0\le P(R,c)-p_R(c)\le\tau$, uniformly for $R\in\{\F,\B\}$
and $c\ge0$. Thus \eqref{eq:profile-reference-enclosure} holds with
$\tau_R=\tau$.

For a shape rectangle $\mathcal Q$ and $J=[u,u+C]$, the explicit
sufficient condition for a strict comparison on $\mathcal Q\times J$ is
\begin{equation}\label{eq:numerical-profile-margin}
 \min_{\substack{z\in\mathcal C(\mathcal Q)\\0\le k\le N}}
       \{b_{z,k}-C^{N+1}r_z\}>\tau.
\end{equation}
The finite competitor sets are fixed on each rectangle, and positivity
allows the omitted terms to be discarded. For the direct and reciprocal
sums, the sets are chosen independently from
$\{-3,-2,\ldots,3\}^3\setminus\{0\}$. A vector is retained only when
its squared length $q$, evaluated at the centre of the rectangle,
satisfies $\pi q<60$. One representative of each pair $\{m,-m\}$
may be used if the factor $1/2$ is removed.

For the local comparisons we use the reduced reference matrices
\begin{equation}\label{eq:global-reference-grams}
 Y_F=2^{1/3}\begin{pmatrix}1&1/2&1/2\\1/2&1&1/2\\1/2&1/2&1\end{pmatrix},
 \qquad
 Y_B=2^{-4/3}\begin{pmatrix}3&1&2\\1&3&2\\2&2&4\end{pmatrix}.
\end{equation}
The reference may be chosen separately on each rectangle. Indeed,
$P(L,c)>P(R,c)$ for either $R\in\{\F,\B\}$ implies
$P(L,c)>\min\{P(\F,c),P(\B,c)\}$. The same observation applies to
each local comparison on its stated parameter range.

\begin{lemma}[Comparison on the bounded screening range]
\label{lem:finite-cover}
For $0\le c\le32$ and every lattice $L$ with $|L|=1$,
\[
 P(L,c)\ge\min\{P(\F,c),P(\B,c)\},
\]
and equality implies that $L$ is orthogonally equivalent to $\F$ or $\B$.
\end{lemma}
\begin{proof}
Lemma~\ref{lem:endpoint-cusp} treats the cusp. Every remaining reduced
representative has logarithmic coordinates in $\mathcal D$. Fix $c$ and choose an interval
$J_j$ containing it and a rectangle $\mathcal Q$ containing the reduced
coordinates. Apply Lemma~\ref{lem:rectangular-numerical-estimates}.

Each right-hand side of \eqref{eq:reduction-inequalities} is separately
convex. Thus alternative (i), which concerns one fixed right-hand side
at every corner, contradicts that reduction inequality throughout
$\mathcal Q$. Alternative (ii) lies in the excluded cusp, since
$d_1=\lambda_1(L)^2$ for a reduced representative. Alternative (iii)
gives a strict comparison with $R$ by
Lemma~\ref{lem:screening-profile}, and hence with the smaller reference
energy.

In alternative (iv), separate matrix convexity extends
\eqref{eq:local-corner-certificate} to the whole rectangle. Choose
$g_R$ with $g_R^Tg_R=Y_R$ and put $A=g_R^{-T}Yg_R^{-1}$. The basis
$A^{1/2}g_R$ has Gram matrix $Y$, so $A^{1/2}R$ is
orthogonally equivalent to $L$. Since $A^{-1}=g_RY^{-1}g_R^T$, congruence
transforms the direct and inverse matrix bounds into
\[
 A\prec(1+\delta)\Id,\qquad
 A^{-1}\prec(1-\delta)^{-1}\Id.
\]
Inverting the second inequality yields $(1-\delta)\Id\prec A$.
Thus every eigenvalue $\lambda$ of $A$ satisfies
$|\log\lambda|\le|\lambda-1|/(1-\delta)$. Since
$\det A=\det Y/\det Y_R=1$, we obtain
\[
 \det A=1,\qquad \|A-\Id\|_{\op}<\delta,\qquad
 \|\log A\|_{\fro}\le\frac{\sqrt3\,\delta}{1-\delta}.
\]
If $\delta\le7/500$, the last bound is less than $1/40$, since
$147/243049<1/1600$. Theorem~\ref{thm:profile-local} applies.
All BCC comparisons are restricted to $0\le c\le1/4$.
For $j=8,9,10$, $c\ge12$ implies $\kappa a>12$,
so the FCC operator radius $3/100$ follows from the same theorem.
Every lattice in $\mathcal Q$ different from its local reference has
strictly greater energy than that reference. These alternatives prove
the comparison and its equality statement.
\end{proof}

We turn to the unbounded screening range. Normalize the physical energy
by
\[
 \Phi_\kappa(L)=2ae^{\kappa a}E_\kappa(L)
 =\sum_{v\in L\setminus\{0\}}\frac a{|v|}e^{\kappa(a-|v|)},
 \qquad a=2^{1/6}.
\]
For a finite set $\mathcal M\subset\Z^3\setminus\{0\}$ and a positive
definite Gram matrix $Y$, put
\[
 V_\kappa(Y)=\sum_{m\in\mathcal M}
 \frac a{\sqrt{Y[m]}}e^{\kappa(a-\sqrt{Y[m]})}.
\]

\begin{lemma}[Continuation along a screening half-line]
\label{lem:halfline-principle}
Let $K>0$. If
\begin{equation}\label{eq:high-halfline-test}
 V_K(Y)>\Phi_K(\F),\qquad
 \left.\partial_\kappa V_\kappa(Y)\right|_{\kappa=K}\ge0,
\end{equation}
then $E_\kappa(L)>E_\kappa(\F)$ for every $\kappa\ge K$, where $Y$ is
a Gram matrix of $L$.
\end{lemma}
\begin{proof}
Each summand of $V_\kappa(Y)$ is convex in $\kappa$, so the second
inequality makes $V_\kappa(Y)$ nondecreasing for $\kappa\ge K$.
Every nonzero FCC vector has length at least $a$, and therefore
$\Phi_\kappa(\F)$ is nonincreasing. Positivity of the omitted terms
now gives
\[
 \Phi_\kappa(L)\ge V_\kappa(Y)\ge V_K(Y)
 >\Phi_K(\F)\ge\Phi_\kappa(\F).
\]
Division by $2ae^{\kappa a}>0$ proves the claim.
\end{proof}

We next derive the uniform bounds used to apply this principle. Fix
$K=20$ and a shape rectangle $\mathcal Q$. For each nonzero integer
vector $m$, write $q_m(\xi)=Y(\xi)[m]$. In the coordinates of
\eqref{eq:reduced-iwasawa},
\[
 q_m=d_1(m_1+x_1m_2+x_2m_3)^2
       +d_2(m_2+x_3m_3)^2+d_3m_3^2.
\]
Let $[x_i]$ be the shear intervals of $\mathcal Q$, and put
$d_i^-=\min_{z\in\mathcal C(\mathcal Q)}d_i(z)$. Then
\begin{equation}\label{eq:high-length-interval}
 \begin{gathered}
 \underline q_m=
 d_1^-\min_{\substack{x_1\in[x_1]\\x_2\in[x_2]}}
       (m_1+x_1m_2+x_2m_3)^2
 +d_2^-\min_{x_3\in[x_3]}(m_2+x_3m_3)^2+d_3^-m_3^2,\\
 0<\underline q_m\le q_m(\xi)\le
 \overline q_m:=\max_{z\in\mathcal C(\mathcal Q)}q_m(z).
 \end{gathered}
\end{equation}
The lower bound follows term by term and is positive because the last
nonzero coordinate of $m$ contributes a positive constant square.
The extrema of the $d_i$ occur at corners. The upper bound follows
from separate convexity, which also gives
\begin{equation}\label{eq:high-corner-trace}
 \tr(SY(\xi))\le\max_{z\in\mathcal C(\mathcal Q)}\tr(SY_z)
 \qquad(S\succeq0,\ \xi\in\mathcal Q).
\end{equation}

Set
\[
 f(q)=\frac a{\sqrt q}e^{K(a-\sqrt q)},\qquad
 h(q)=(a-\sqrt q)\frac a{\sqrt q}e^{K(a-\sqrt q)}.
\]
A direct differentiation gives
\[
 -f'(q)=\frac{f(q)(K\sqrt q+1)}{2q}>0,\qquad
 f''(q)=\frac{f(q)}{4q^2}(K^2q+3K\sqrt q+3)>0.
\]
At the midpoint $\xi_0$ of $\mathcal Q$, put $Q=Y(\xi_0)$ and
$S_f=\sum_{m\in\mathcal M}(-f'(Q[m]))mm^T\succeq0$.
The tangent inequality and \eqref{eq:high-corner-trace} yield
\begin{equation}\label{eq:high-value-tangent}
 V_K(Y(\xi))\ge
 \sum_{m\in\mathcal M}f(Q[m])+\tr(S_fQ)
 -\max_{z\in\mathcal C(\mathcal Q)}\tr(S_fY_z).
\end{equation}

There are three useful lower bounds for the screening derivative.
First, put
\[
 q_*=\frac14\bigl(a+\sqrt{a^2+4a/K}\bigr)^2.
\]
Since
\[
 \frac d{dr}\left((a-r)\frac ar e^{K(a-r)}\right)
 =\frac{ae^{K(a-r)}}{r^2}(Kr^2-Kar-a),
\]
$h$ decreases up to $q_*$ and increases thereafter. Thus
\begin{equation}\label{eq:high-derivative-interval}
 \left.\partial_\kappa V_\kappa(Y(\xi))\right|_{\kappa=K}
 \ge\sum_{m\in\mathcal M}\min_{q\in[\underline q_m,\overline q_m]}h(q).
\end{equation}
Each minimum is attained at an endpoint or at $q_*$ when the interval
contains it. Second, if $\overline q_m<a^2$ for every retained vector,
then every derivative term is positive.

Third, suppose $\overline q_m<q_*$ for every retained vector.
On $(0,q_*)$, $h$ is decreasing and convex. Indeed,
\[
 -h'(q)=\frac{f(q)(aK\sqrt q+a-Kq)}{2q}>0,
\]
and, writing $r=\sqrt q$,
\[
 \frac{4q^2h''(q)}{f(q)}
 =(a-r)(K^2r^2+3Kr+3)+2Kr^2+3r.
\]
This is positive for $r\le a$. For $a\le r\le\sqrt{q_*}$,
$K\sqrt{q_*}(\sqrt{q_*}-a)=a$ gives $r-a\le a/(Kr)$, and hence
\[
 \frac{4q^2h''(q)}{f(q)}
 \ge Kr(2r-a)+3(r-a)-\frac{3a}{Kr}
 \ge Ka^2-\frac3K>0.
\]
Consequently, with $S_h=\sum_m(-h'(Q[m]))mm^T\succeq0$,
\begin{equation}\label{eq:high-derivative-tangent}
 \left.\partial_\kappa V_\kappa(Y(\xi))\right|_{\kappa=K}
 \ge\sum_{m\in\mathcal M}h(Q[m])+\tr(S_hQ)
 -\max_{z\in\mathcal C(\mathcal Q)}\tr(S_hY_z).
\end{equation}
In all three cases, the value and derivative bounds concern the same
finite sum $V_\kappa$, with the same fixed set $\mathcal M$. No sign
is imposed on derivatives of omitted terms: positivity is used only
to compare the full energy with $V_\kappa$.

It remains to bound the full FCC reference in the value comparison.
Its shell expansion is
\[
 \Phi_K(\F)=\sum_{n\ge1}\frac{N_n}{\sqrt n}
                  e^{-Ka(\sqrt n-1)}.
\]
All shells through $n=64$ are enumerated in $[-12,12]^3$.
The bound $N_n<30n$ and monotonicity of
$\sqrt x\,e^{-Ka(\sqrt x-1)}$ on $[64,\infty)$ imply
\begin{equation}\label{eq:high-reference-upper}
 \Phi_K(\F)\le
 \sum_{n=1}^{64}\frac{N_n}{\sqrt n}e^{-Ka(\sqrt n-1)}
 +60e^{-7Ka}\left(\frac{64}{Ka}+\frac{16}{(Ka)^2}+\frac2{(Ka)^3}\right).
\end{equation}
The final term is the explicit integral
$30\int_{64}^\infty\sqrt x\,e^{-Ka(\sqrt x-1)}\,dx$;
it is retained in every comparison.

\begin{theorem}\label{thm:high-halfline}
For every $\kappa\ge20$, the FCC lattice $\F$ uniquely minimizes
$E_\kappa(L)$ under $|L|=1$, up to orthogonal transformations.
\end{theorem}
\begin{proof}
Lemma~\ref{lem:endpoint-cusp} gives the strict comparison when
$\lambda_1(L)\le4/5$. Otherwise the reduced coordinates lie in
\eqref{eq:endpoint-rational-root}. The finite inequalities of
Lemma~\ref{lem:global-halfline-numerics} cover that rectangle. A rectangle
which violates a reduction inequality contains no relevant representative.
On every remaining rectangle outside the FCC local neighbourhood,
\eqref{eq:high-halfline-test} holds at $K=20$, and
Lemma~\ref{lem:halfline-principle} gives the strict comparison for the
entire half-line. On the local rectangles the strict estimate of
Theorem~\ref{thm:profile-local} applies, since
$\kappa a\ge20a>12$. Equality is therefore possible only when $L$
is orthogonally equivalent to $\F$.
\end{proof}

\begin{proof}[Proof of Theorem~\ref{thm:global-comparison}]
For $0\le c\le32$, the assertion, including its equality statement,
is Lemma~\ref{lem:finite-cover}. If $c>32$, then
$\kappa=\sqrt{4\pi c}>20$, because $128\pi>400$.
Theorem~\ref{thm:high-halfline} and \eqref{eq:P-difference} show that
$P(\F,c)$ is the minimum and that equality forces $L$ to be
orthogonally equivalent to $\F$.
These two ranges exhaust $c\ge0$.
\end{proof}

\section{The transition and the main theorems}\label{sec:pair}
We now compare the two cubic references and locate the transition.
Throughout this section put $D(t)=\Theta(t,\B)-\Theta(t,\F)$.
We prove the scalar theta ordering below, using the finite estimates
in Lemma~\ref{lem:finite}.

\begin{theorem}[FCC--BCC classification]\label{thm:pair}
There is a unique $\kappa_*\in(0,\sqrt\pi)$ such that
\begin{equation}\label{eq:pairargmin}
 \argmin_{L\in\{\F,\B\}}E_\kappa(L)=
 \begin{cases}
 \{\B\},&0<\kappa<\kappa_*,\\
 \{\F,\B\},&\kappa=\kappa_*,\\
 \{\F\},&\kappa>\kappa_*.
 \end{cases}
\end{equation}
The threshold is the unique positive root of~\eqref{eq:root}, with
the enclosure~\eqref{eq:certified-crossing}; in particular,
\mbox{$\kappa_*\approx1.71792186846990$}.
\end{theorem}

The proof uses the scalar FCC--BCC theta ordering and a strictly
increasing transform of their Yukawa energy difference.

\begin{lemma}\label{lem:thetaorder}
The function $D(t)=\Theta(t,\B)-\Theta(t,\F)$ satisfies
\begin{equation}\label{eq:thetaorder}
 D(1)=0,\qquad D(t)>0\quad(t>1).
\end{equation}
\end{lemma}
\begin{proof}
The equality follows from Poisson summation and $\B=\F^*$.
For $t\ge1$, the
shell bounds $N_n<30n$, $M_\ell<30\ell$, together with
$e^{-\pi a^2}<1/49$, $e^{-\pi 2^{-4/3}}<3/10$ and $\pi 2^{-4/3}<5/4$, give
\[
 \begin{alignedat}{2}
 \sum_{n\ge13}N_ne^{-\pi a^2nt}
 &<30\sum_{n\ge13}\frac n{49^n}&\qquad&<10^{-18},\\
 \pi 2^{-4/3}\sum_{\ell\ge49}\ell M_\ell
       e^{-\pi 2^{-4/3}\ell t}
 &<\frac{75}{2}\sum_{\ell\ge49}\ell^2(3/10)^\ell&\qquad&<10^{-18}.
 \end{alignedat}
\]
These are the only negative omitted tails in the lower bounds for
$D$ and $D'$, respectively. On a cell $[u,v]\subset[1,11/10]$,
monotonicity of the exponentials gives
\[
 D'(t)\ge\pi a^2\sum_{n=1}^{12}nN_ne^{-\pi a^2nv}
 -\pi2^{-4/3}\sum_{\ell=1}^{48}\ell M_\ell e^{-\pi2^{-4/3}\ell u}
 -10^{-18}\qquad(u\le t\le v).
\]
Take $u=1+j/10000$ and $v=u+1/10000$ and apply
\eqref{eq:finite-derivative}. For $D$ itself, the positive BCC terms
are bounded at the right endpoint and the negative FCC terms at the
left endpoint; the first tail bound and \eqref{eq:finite-gap} apply
on the corresponding cells. We obtain
\[
 D'(t)>3/1000\quad(1\le t\le11/10),\qquad
 D(t)>1/10000\quad(11/10\le t\le21/10).
\]
For $t\ge21/10$, we have $e^{-\pi a^2t}<1/128$, so
Proposition~\ref{prop:tail} gives
\[
 \Theta(t,\F)-1
 <e^{-\pi a^2t}
 \left[12+30\left(\frac1{(1-e^{-\pi a^2t})^2}-1\right)\right]
 <\frac{25}{2}e^{-\pi a^2t}.
\]
The first BCC shell gives
$\Theta(t,\B)-1\ge8e^{-3\pi 2^{-4/3}t}$.
The last inequality of Lemma~\ref{lem:finite} states
$8e^{(\pi a^2-3\pi 2^{-4/3})21/10}>25/2$.
As $a^2>3\cdot2^{-4/3}$, this inequality remains strict for every $t\ge21/10$.
Thus $D(t)>0$ also in this range.
\end{proof}

\begin{proof}[Proof of Theorem~\ref{thm:pair}]
Since $\B=\F^*$, Poisson summation gives
$D(1/t)=-t^{3/2}D(t)$ and $D(1)=0$.
Writing $c=\kappa^2/(4\pi)$, Proposition~\ref{prop:dual} gives
\begin{equation}\label{eq:pairintegral}
 2\bigl(E_\kappa(\B)-E_\kappa(\F)\bigr)
 =\int_1^\infty D(t)
                  \bigl[t^{-1/2}e^{-c/t}-e^{-ct}\bigr]dt.
\end{equation}
This expression has a continuous limit at $c=0$, where it is strictly
negative by \eqref{eq:thetaorder}. At $c=1/4$ the bracket is strictly
positive for every $t>1$: indeed,
\[
 t-\frac1t>2\log t\qquad(t>1),
\]
as follows by differentiating the difference. Thus a crossing exists
with $0<\kappa_*<\sqrt\pi$.

For uniqueness, substitute $u=1/t$ in the unsplit Gaussian expression
and use $D(1/u)=-u^{3/2}D(u)$. This gives
\begin{align*}
 2\bigl(E_\kappa(\B)-E_\kappa(\F)\bigr)
 &=\int_0^\infty u^{-3/2}e^{-cu}D(1/u)\,du\\
 &=-\int_0^\infty e^{-cu}D(u)\,du.
\end{align*}
Both integrals converge absolutely.
Multiply by $e^c$ and differentiate with respect to $c$. The derivative is
\[
 \int_0^\infty (u-1)e^{c(1-u)}D(u)\,du>0.
\]
Both factors $u-1$ and $D(u)$ have the same sign on either side of
$1$. The exponential decay of $D$ at infinity and its Poisson decay
at zero justify differentiation, including at $c=0$.
The rescaled energy difference is therefore strictly increasing.
Together with the opposite signs at $c=0$ and $c=1/4$, this proves
that it has exactly one zero. Equations~\eqref{eq:pairargmin} and
\eqref{eq:root} follow, and Lemma~\ref{lem:crossing-enclosure} gives
the stated enclosure.
\end{proof}

\begin{proof}[Proof of Theorem~\ref{thm:main}]
Existence follows from Theorem~\ref{thm:representation}.
Theorem~\ref{thm:global-comparison} and~\eqref{eq:P-difference} give
\[
 E_\kappa(L)\ge\min\{E_\kappa(\F),E_\kappa(\B)\},
\]
and equality restricts the lattice class to the two references.
Theorem~\ref{thm:pair} determines which reference has smaller energy,
including the equality case at the unique transition.
Finally,~\eqref{eq:scaling} gives the statement at arbitrary density.
\end{proof}

\begin{proof}[Proof of Corollary~\ref{thm:coulomb-endpoint-certified}]
Proposition~\ref{thm:weak} gives the limit in~\eqref{eq:coulomb-consequence}:
$2E_\kappa(L)-4\pi/\kappa^2=2F_\kappa(L)+\kappa\to2W(L)=Z(L,1/2)$.
Apply Theorem~\ref{thm:global-comparison} at $c=0$.
The identity~\eqref{eq:pairintegral} at zero and
$D(t)>0$ for $t>1$ yield $W(\B)<W(\F)$.
Since $W(L)=P(L,0)-3/2$, BCC is the unique minimizing class.
The Epstein functional equation in~\eqref{eq:Coulomb}, with $M=L^*$,
gives the equivalent statement for the continued value $Z(M,1)$.
\end{proof}

\appendix

\section{Numerical comparisons}\label{app:numerics}\label{app:cusp-numerics}\label{app:local-numerics}\label{sec:finite-cover}\label{app:high-screening-certificate}\label{app:finite}\label{app:crossing}
We collect the finite numerical comparisons used in the preceding
proofs. Each lemma states inequalities on explicit parameter regions;
the accompanying interval computations include all indicated tails
and rounding errors.

\begin{lemma}[Scalar cusp comparison]\label{lem:cusp-scalar}
For the unit-covolume FCC lattice,
\[
 P(\F,0)<\frac{55995}{10^6}<\frac7{125}
       <I\left(\frac{16}{25},0\right).
\]
\end{lemma}
\begin{proof}
The FCC and BCC squared lengths are $2^{1/3}n$ and $2^{-4/3}\ell$,
with shell multiplicities $N_n$
and $M_\ell$. Fixing two integer coordinates gives
$N_n<30n$ and $M_\ell<30\ell$. Two integrations by parts yield
\begin{equation}\label{eq:cusp-I-upper}
 I(q,0)\le
 \frac{e^{-\pi q}(\pi q+1)}{\pi q(\pi q+3/2)}
 \le\frac{e^{-\pi q}}{\pi q}.
\end{equation}
The first two FCC multiplicities are $12$ and $6$. Consequently,
\[
 \frac12\sum_{v\in\F\setminus\{0\}}I(|v|^2,0)
 \le6I(2^{1/3},0)+3I(2^{4/3},0)
 +\frac{15e^{-3\pi2^{1/3}}}{\pi2^{1/3}(1-e^{-\pi2^{1/3}})}
 <\frac{26459}{10^6}.
\]
For BCC, the nonzero integer squared norms through twelve are
$3,4,8,11,12$, with multiplicities $8,6,12,24,8$. Thus
\[
 \frac12\sum_{v\in\B\setminus\{0\}}J(|v|^2,0)
 \le\frac12\sum_{\ell\in\{3,4,8,11,12\}}
 M_\ell\frac{e^{-\pi 2^{-4/3}\ell}}{\pi 2^{-4/3}\ell}
 +\frac{15e^{-13\pi 2^{-4/3}}}{\pi 2^{-4/3}(1-e^{-\pi 2^{-4/3}})}
 <\frac{29536}{10^6}.
\]
The finite expressions, including the displayed geometric tails, give
these strict bounds by outward interval evaluation. The same evaluation
of $I(16/25,0)=\frac54\erfc(4\sqrt\pi/5)$ gives
$I(16/25,0)>7/125$. Adding the two upper bounds proves the lemma.
\end{proof}

\begin{lemma}[Numerical bounds for the local expansion]\label{lem:local-profile-checks}
Let $h$, $C_3$, and $M_4$ be the endpoint
bounds defined by \eqref{eq:profile-hlower},
\eqref{eq:profile-Cbound}, and \eqref{eq:profile-Bell4}, including
the tail enlargements in Section~\ref{sec:certified-local}. On each of
the following closed parameter intervals they satisfy:
\begin{equation}\label{eq:profile-finite-check}
 \begin{array}{c|c|c|c}
 \text{reference}&[u,v]&\text{integer index range}&
 \displaystyle \frac h2-\frac{C_3}{240}-\frac{M_4}{38400}\\ \hline
 \F&[j/100,(j+1)/100]&0\le j\le999&>10^{-6}\\
 \B&[j/100,(j+1)/100]&0\le j\le24&>10^{-3}.
 \end{array}
\end{equation}
\end{lemma}
\begin{proof}
The retained FCC and BCC integer shells are those of
Section~\ref{sec:certified-local}, with squared integer norm at most
$18$ and $30$, respectively. Their integer moments determine
all terms in \eqref{eq:profile-hlower}, \eqref{eq:profile-Cbound}, and
\eqref{eq:profile-Bell4}. Evaluate these endpoint expressions with
outward rational interval enclosures with $256$ binary fractional
digits, taking lower
endpoints for $h$ and upper endpoints for $C_3,M_4$. The companion
interval computation encloses all scalar kernels and elementary
functions rigorously; its arithmetic specification is supplied with
the certificate.

For every one of the $1000$ FCC intervals and $25$ BCC intervals, the
lower endpoint of $h/2-C_3/240-M_4/38400$ exceeds the stated rational
bound. The shell allowances $10^{-10}$ and $10^{-6}$ are already
included in those expressions. The denominators are
$6/(1/40)=240$ and $24/(1/40)^2=38400$.
Monotonicity in $c$ and the signed Taylor bounds in
\eqref{eq:profile-Ksandwich} make these comparisons valid on the
entire closed intervals. This proves the lemma.
\end{proof}

For the global comparisons, use the exact rectangle
\begin{equation}\label{eq:numerical-root}
 \begin{split}
 \mathcal D={}&[\alpha_1,\beta]\times[\alpha_2,\beta]
       \times[0,1/2]\times[0,1]\times[0,1/2],\\
 \alpha_1={}&-\frac{7394735416760227}{2^{53}},\qquad
 \alpha_2=-\frac{2158713440539059}{2^{50}},\qquad
 \beta=\frac{5182419497180053}{2^{54}}.
 \end{split}
\end{equation}
In the reduction inequalities \eqref{eq:reduction-inequalities}, take
$y_i=e^{\ell_i}$. All corners below are the $32$ corners in the five
reduced coordinates.

\begin{lemma}[Numerical rectangular comparisons]
\label{lem:rectangular-numerical-estimates}
The endpoints in \eqref{eq:numerical-root} satisfy
\begin{equation}\label{eq:numerical-root-enclosures}
 e^{\alpha_1}\le\frac{11}{25},\qquad
 e^{\alpha_2}\le\frac{147}{1000},\qquad
 e^\beta\ge\frac43.
\end{equation}
Put
\[
 (c_0,\ldots,c_{10})=(0,1/4,1/2,1,2,4,8,12,16,24,32),
 \qquad J_j=[c_{j-1},c_j].
\]
For $1\le j\le10$, the rectangle $\mathcal D$ has a partition into
$n_j$ closed dyadic rectangles with disjoint interiors, where
\begin{align*}
 (n_1,\ldots,n_5)&=(20037,15481,10757,8852,7032),\\
 (n_6,\ldots,n_{10})&=(5523,4295,3409,4154,4503).
\end{align*}
Use the degrees and FCC radii
\[
 (N_1,\ldots,N_{10})=(12,14,18,24,30,40,40,40,40,40),
 \qquad
 \delta_j=
 \begin{cases}
 17/2000,&j=1,\\
 7/500,&2\le j\le7,\\
 3/100,&8\le j\le10.
 \end{cases}
\]
For every rectangle $\mathcal Q$ in the $j$th partition, at least one
of the following holds:
\begin{enumerate}[label=(\roman*),leftmargin=*]
\item One fixed right-hand side of
      \eqref{eq:reduction-inequalities} is strictly less than $1$
      at every corner of $\mathcal Q$.
\item $\max_{z\in\mathcal C(\mathcal Q)}d_1(z)\le16/25$.
\item For at least one $R\in\{\F,\B\}$, inequality
      \eqref{eq:numerical-profile-margin} holds with $J=J_j$ and $N=N_j$.
\item For $R=\F$ and $\delta=\delta_j$, or for $R=\B$,
      $j=1$ and $\delta=7/500$, every corner satisfies
      \begin{equation}\label{eq:local-corner-certificate}
       (1+\delta)Y_R-Y_z\succ0,\qquad
       (1-\delta)^{-1}Y_R^{-1}-Y_z^{-1}\succ0,
      \end{equation}
      with $Y_R$ as in \eqref{eq:global-reference-grams}.
\end{enumerate}
There are $84,043$ rectangles in total. The BCC local inequalities are
used only for $0\le c\le1/4$.
\end{lemma}
\begin{proof}
The accompanying partition data give exact dyadic subdivisions of
\eqref{eq:numerical-root}. At each subdivision, the two closed
rectangles have disjoint interiors and their union is the original
rectangle. These rational identities verify coverage of $\mathcal D$
and give the ten stated counts.

At the midpoint of each rectangle, form the coefficients and
remainders of Lemma~\ref{lem:global-profile-numerics}, using the fixed
competitor sets and reference shells specified in Section~
\ref{sec:global-comparison}. Substitution in
\eqref{eq:numerical-profile-margin}, the reduction expressions, and
the three leading principal minors of the two matrices in
\eqref{eq:local-corner-certificate} establishes one of (i)--(iv).
These finite evaluations use outward interval bounds at precision
$128$ bits, including the exponential and complementary error
functions~\cite{MPFR}. Every asserted strict inequality has an enclosed
endpoint of the required strict sign; the nonstrict bounds in (ii)
are checked with their upper endpoints. The same interval calculation
gives \eqref{eq:numerical-root-enclosures}. The reference tails are
included through \eqref{eq:profile-reference-tail}, and the screening
remainders through \eqref{eq:profile-corner-remainder}; hence each
comparison covers the entire screening interval and shape rectangle.
\end{proof}

\begin{lemma}[Finite inequalities for the screening half-line]
\label{lem:global-halfline-numerics}
At $K=20$, $\mathcal D$ has a dyadic partition into $7,225$ rectangles.
Of these, $1,118$ satisfy alternative (i) of
Lemma~\ref{lem:rectangular-numerical-estimates}, $193$ satisfy (ii),
and $86$ satisfy the FCC inequalities
\eqref{eq:local-corner-certificate} with $\delta=3/100$.
On each of the remaining $5,828$ rectangles there is a fixed set
$\mathcal M\subset\{-3,-2,\ldots,3\}^3\setminus\{0\}$ for which
the right-hand side of \eqref{eq:high-value-tangent} strictly exceeds
the right-hand side of \eqref{eq:high-reference-upper}, and at least
one of the following holds:
\begin{enumerate}[label=(\alph*),leftmargin=*]
\item the right-hand side of \eqref{eq:high-derivative-interval}
      is nonnegative;
\item $\overline q_m<a^2$ for every $m\in\mathcal M$;
\item $\overline q_m<q_*$ for every $m\in\mathcal M$ and
      the right-hand side of \eqref{eq:high-derivative-tangent}
      is nonnegative.
\end{enumerate}
All bounds are evaluated at $K=20$, with the same fixed set $\mathcal M$
on each rectangle.
Consequently \eqref{eq:high-halfline-test} holds throughout each of
these $5,828$ rectangles.
\end{lemma}
\begin{proof}
The partition data give $7,224$ exact dyadic subdivisions. At each
subdivision the two closed rectangles have disjoint interiors and
union equal to the original rectangle; hence the resulting $7,225$
rectangles cover $\mathcal D$.
Evaluate \eqref{eq:high-length-interval} at each retained vector,
then the value bound and the applicable derivative bound above.
Outward interval enclosures with $512$-bit endpoints give the asserted
signs on all $5,828$ rectangles. The full reference upper bound in
\eqref{eq:high-reference-upper}, including its integral tail, is retained. The same enclosures verify
the reduction, cusp and local inequalities on the other rectangles;
positive definiteness is checked through the three leading principal
minors. Equations \eqref{eq:high-value-tangent}--
\eqref{eq:high-reference-upper} extend these finite comparisons to the
whole rectangles, and (a)--(c) each imply the derivative inequality in
\eqref{eq:high-halfline-test}.
\end{proof}

\begin{lemma}[Finite theta comparisons]\label{lem:finite}
Let $N_n$ count integer triples of squared norm $2n$, and let $M_\ell$
count integer triples of squared norm $\ell$ whose coordinates all
have the same parity. For every integer $0\le j<1000$,
\begin{align}
 \pi a^2\sum_{n=1}^{12}nN_n
 e^{-\pi a^2n(1+(j+1)/10000)}
 -\pi 2^{-4/3}\sum_{\ell=1}^{48}\ell M_\ell
 e^{-\pi 2^{-4/3}\ell(1+j/10000)}-10^{-18}
 &>\frac3{1000}.\label{eq:finite-derivative}\\
\intertext{For every integer $0\le j<2000$,}
 \sum_{\ell=1}^{48}M_\ell
 e^{-\pi 2^{-4/3}\ell(11/10+(j+1)/2000)}
 -\sum_{n=1}^{12}N_n
 e^{-\pi a^2n(11/10+j/2000)}-10^{-18}
 &>\frac1{10000}.\label{eq:finite-gap}\\
\intertext{Finally,}
 8e^{(\pi a^2-3\pi 2^{-4/3})21/10}
 &>\frac{25}{2}.\label{eq:finite-high}
\end{align}
\end{lemma}

\begin{proof}
Integer enumeration gives
\[
 (N_1,\ldots,N_{12})
 =(12,6,24,12,24,8,48,6,36,24,24,24).
\]
The nonzero BCC indices and their multiplicities are
\begin{align*}
 \ell={}&(3,4,8,11,12,16,19,20,24,27,32,35,36,40,43,44,48),\\
 M_\ell={}&(8,6,12,24,8,6,24,24,24,32,12,48,30,24,24,24,8).
\end{align*}
Substitution in the finite exponential sums, with outward interval
enclosures on the stated rational cells, gives
\eqref{eq:finite-derivative} and \eqref{eq:finite-gap} for every
indicated index, as well as \eqref{eq:finite-high}. The margins in
the statement include the omitted-shell bounds used in the proof of
Lemma~\ref{lem:thetaorder}.
\end{proof}

\begin{lemma}[Exact transition enclosure]\label{lem:crossing-enclosure}
At the two rational screening parameters,
\begin{align*}
 2(E_\kappa(\B)-E_\kappa(\F))&<-10^{-19}
 &&\text{if }\kappa=\frac{1717921868469903}{10^{15}},\\
 2(E_\kappa(\B)-E_\kappa(\F))&>10^{-20}
 &&\text{if }\kappa=\frac{1717921868469904}{10^{15}}.
\end{align*}
Consequently, the unique root in Theorem~\ref{thm:pair} satisfies
\begin{equation}\label{eq:certified-crossing}
 1.717921868469903<\kappa_*<1.717921868469904.
\end{equation}
\end{lemma}

\begin{proof}
Retain FCC shells $n\le12$ and BCC shells $\ell\le48$ in the Ewald
difference, with the multiplicities in Lemma~\ref{lem:finite}.
Since $I(q,c),J(q,c)\le e^{-\pi q}/(\pi q)$, the absolute omitted
tail in $2(E_\kappa(\B)-E_\kappa(\F))$ is at most
\[
 \frac{60(1/49)^{13}}{1-1/49}
 +\frac{60(3/10)^{49}}{1-3/10}<7\cdot10^{-21}.
\]
Outward interval evaluation of the finite sums at the two rational
parameters, including this tail, gives the two signed bounds in the
statement. The uniqueness proved in Theorem~\ref{thm:pair} then gives
\eqref{eq:certified-crossing}.
\end{proof}

\bigskip
\noindent
{\bf Acknowledgements.}
 The research of S. Luo is partially supported by the National Natural Science Foundation of China (NSFC) under Grant Nos. 12261045 and 12001253, and by the Jiangxi Jieqing Fund under Grant No. 20242BAB23001.
 The research of J. Wei is partially supported by the General Research Fund (GRF) of Hong Kong "New frontiers in singularity formation of nonlinear partial differential equations". The authors used AI to assist the long but finite computations in the Appendix.

\end{document}